\documentclass{amsart}
\usepackage{amsmath,amsthm,amsfonts,amssymb,amscd,overpic,mathrsfs}
\usepackage{dsfont}
\usepackage{amsthm}
\usepackage{amssymb}
\usepackage{graphicx}
\usepackage{amstext,amsthm,amsfonts,amsmath}
\usepackage[T1]{fontenc}
\usepackage{enumerate}
\usepackage{tikz}
\usepackage{pgfplots, pgfplotstable}
\usepackage[font=scriptsize]{caption}
\usepackage{indentfirst}
\usepackage{subfigure}

\newcommand{\mfb}{\mathfrak{b}}

\newcommand{\mcalB}{{\mathcal{B}}}
\newcommand{\mcalH}{{\mathcal{H}}}

\newcommand{\mfT}{{\mathfrak{T}}}

\newcommand{\Si}{{\Sigma}}

\newcommand{\eqdef}{\stackrel{\scriptscriptstyle\rm def}{=}}
\definecolor{Red}{cmyk}{0,1,1,0}

\definecolor{verde}{cmyk}{1,0,1,0}

\definecolor{azul}{cmyk}{1,1,0,0}

\begin{document}

\title[Non-hyperbolic Iterated Function Systems]{Non-hyperbolic Iterated Function Systems:\\ 
attractors and stationary measures}

\author[E. Matias]{Edgar Matias}
\email{edgar@mat.puc-rio.br}
\author[L. J.~D\'\i az]{Lorenzo J. D\'\i az}
\email{lodiaz@mat.puc-rio.br}
\address{Departamento de Matem\'atica PUC-Rio, Marqu\^es de S\~ao Vicente 225, G\'avea, Rio de Janeiro 22451-900, Brazil}

\begin{abstract}
We consider iterated function systems $\mathrm{IFS}(T_1,\dots,T_k)$ consisting of continuous self maps of a compact metric space $X$.
We introduce the subset $S_{\mathrm{t}}$ of 
{\emph{weakly hyperbolic sequences}}
$\xi=\xi_0\ldots\xi_n \ldots \in \Sigma_k^+$ having the property that 
$\bigcap_n T_{\xi_{0}}\circ\cdots\circ T_{\xi_{n}}(X)$ is a point $\{\pi(\xi)\}$.
The  {\emph{target set}}   $\pi(S_{\mathrm{t}})$  
 plays a role  similar to the semifractal introduced by
 Lasota-Myjak. 
 
Assuming that $S_{\mathrm{t}}\ne \emptyset$ (the only hyperbolic-like condition we assume) we prove that the IFS has  at most one strict attractor and we state a sufficient condition guaranteeing 
that the strict attractor is the closure of the target set. 
Our approach applies to a large class of genuinely non-hyperbolic IFSs (e.g. with maps with expanding fixed points) and provides a necessary and sufficient condition for the existence of a globally attracting fixed point of the Barnsley-Hutchinson operator.
We provide sufficient conditions under which the disjunctive chaos game yields the target set (even when it
 is not a strict attractor).
 
 We  state a sufficient condition for the asymptotic stability 
 of the Markov operator of a recurrent IFS. 
 For IFSs defined on $[0,1]$ we
 give a simple condition for their
  asymptotic stability.  In the 
 particular case of IFSs with probabilities satisfying 
 a  ``locally injectivity'' condition,
 we prove that 
 if the target set has at least two elements 
 then the  Markov operator is 
 asymptotically stable and
 its stationary measure is supported 
  in the closure 
 of the target set.
 \end{abstract}

\begin{thanks}{This paper is part of the PhD thesis of EM (PUC-Rio) supported by CAPES.
EM thanks the hospitality of Centro de Matem\'atica of Univ. of Porto for its hospitality and the
partial support of 
EU Marie-Curie IRSES ``Brazilian-European partnership in Dynamical
Systems" (FP7-PEOPLE-2012-IRSES 318999 BREUDS). 
LJD is partially supported by CNPq and CNE-Faperj.
The authors warmly thank P. Barrientos and K. Gelfert for their useful comments on this paper.  
}\end{thanks}

\date{}
\newtheorem{mteo}{Theorem}
\newtheorem{mcoro}{Corollary}
\newtheorem{mprop}{Proposition}

\newtheorem{teo}{Theorem}[section]
\newtheorem{prop}[teo]{Proposition}
\newtheorem{lema}[teo]{Lemma}
\newtheorem{schol}[teo]{Scholium}
\newtheorem{coro}[teo]{Corollary}
\newtheorem{defi}[teo]{Definition}
\newtheorem{example}[teo]{Example}
\newtheorem{remark}[teo]{Remark}
\newtheorem{nota}[teo]{Notation}
\newtheorem{claim}[teo]{Claim}
\newtheorem{fact}[teo]{Fact}

\numberwithin{equation}{section}

\maketitle




\section{Introduction}


In this paper we study iterated function systems (IFSs) associated to continuous self-maps 
$T_1, \dots, T_k$, $k\ge 2$,
defined on a compact
metric space $(X,d)$ (denoted by $\mathrm{IFS}(T_1,\dots,T_k)$).
In his fundamental paper \cite{Hu}, Hutchinson considered hyperbolic (uniformly contracting) IFSs and proved 
 the existence and uniqueness  of global  attractors and stationary measures for such IFSs. The aim of
this paper is to obtain similar results for genuinely non-hyperbolic IFSs having
contracting and expanding regions as well as contracting and expanding fixed points.

A key ingredient in this study is the
so-called
 {\emph{Barnsley-Hutchinson operator}}
 of an IFS  $\mathfrak{F}=\mathrm{IFS}(T_{1},\dots T_{k})$
 that associates to each subset $A$ of $X$ the set
\begin{equation}\label{e.BH}
\mathcal{B}_{\mathfrak{F}}(A)\eqdef  \bigcup_{i=1}^{k} T_{i}(A).
\end{equation}
This operator acts continuously in the space of non-empty compact subsets 
of $X$ endowed with the Hausdorff 
metric.
In the hyperbolic setting (all maps $T_i$ are uniform contractions) the operator 
$\mathcal{B}_{\mathfrak{F}}$
 has a unique  global attractor: there exists a compact set $A_\mathfrak{F}$,
 called the
{\emph{attractor of the IFS,}}
such that
 $$
\lim_{n\to\infty} \mathcal{B}_{\mathfrak{F}}^{n}(K)= A_{\mathfrak{F}}
 \quad
 \mbox{for every  compact set $K\subset X, \, K\ne\emptyset$,}
 $$
see  \cite{Hu}.  
Edalat~\cite{Ed}  extended 
this result to \emph{weakly hyperbolic IFSs}, that is, IFSs satisfying  the following
``reverse'' contracting condition
\begin{equation}\label{e.edalat}
\mathrm{diam}\,\big( T_{\xi_{0}}\circ\cdots\circ T_{\xi_{n}}(X)\big)\to 0
\quad \mbox{for every }\,\, \xi=\xi_0\xi_1\xi_2\dots \in\Sigma_k^+,
\end{equation}
where  $\Sigma_k^+\eqdef\{1,\dots,k\}^\mathbb{N}$.

In this paper we will study a more general setting than the above one, considering genuinely non-hyperbolic IFSs. 
One of our  goals is to describe the global and local ``attractors'' of $\mathcal{B}_{\mathfrak{F}}$.
More precisely, we will consider so-called  {\emph{strict}} and \emph{Conley attractors}.
A compact set  $A\subset X$ is a  \emph{strict attractor}  
of the IFS ${\mathfrak{F}}$
 if there is an open neighbourhood
 $U$ of $A$ such that 
 $$
\lim_{n\to\infty} \mathcal{B}_{\mathfrak{F}}^{n}(K)= A
 \quad
 \mbox{for every  compact set $K\subset U$.}
 $$
The \emph{basin of attraction} of $A$ is the largest open neighbourhood of $A$ for which the above  property holds. A strict attractor whose basin of attraction is the whole space is  a 
{\emph{global attractor.}}
  A compact set $S\subset X$ is a \emph{Conley attractor} of  the IFS ${\mathfrak{F}}$
   if there exists  an open neighbourhood $U$ of $S$ such that
 $$
 \lim_{n\to \infty} \mathcal{B}_{\mathfrak{F}}^{n}(\overline{U})=S.
 $$
  The continuity of the Barnsley-Hutchinson operator  $\mathcal{B}_{\mathfrak{F}}$
  implies that Conley  
  and strict attractors both are fixed points of  $\mathcal{B}_{\mathfrak{F}}$.
Note also that strict attractors are Conley attractors but the converse is not true in general.
Finally, we say that the IFS $\mathfrak{F}$ is 
\emph{asymptotically stable}  if there is a (unique) global attractor.

%

The above mentioned results in \cite{Hu,Ed} require some sort of global contraction (hyperbolicity) of the IFS. 
Having in mind the definition of weakly hyperbolicity in \eqref{e.edalat}, we introduce 
 the subset $S_{\mathrm{t}}\subset\Sigma_k^+$ of
 {\emph{weakly hyperbolic sequences}} defined by
\begin{equation}\label{e.stweakly}
S_{\mathrm{t}}\eqdef\big\{\xi\in \Sigma^{+}_k\colon 
 \displaystyle\lim_{n\to \infty}\mathrm{diam}\big(T_{\xi_{0}}\circ\cdots\circ T_{\xi_{n}}(X)\big)=0\big\}.
\end{equation}
Note that for a weakly hyperbolic IFS one has $S_{\mathrm{t}}=\Sigma^{+}_k$. If $S_{\mathrm{t}}\ne\Sigma^{+}_k$ we will call the IFS \emph{non-weakly hyperbolic}.
We say that an IFS 
 {\emph{has a weakly hyperbolic sequence}} if $S_{\mathrm{t}}\ne \emptyset$. 
 When $S_{\mathrm{t}}\ne \emptyset$ then it contains a residual subset of 
 $\Sigma_k^+$\footnote{This follows using genericity standard arguments, see for instance the construction
 in \cite[Proposition 3.15]{DiGe}}.
 We replace the
  condition {\emph{every}} sequence is weakly hyperbolic by the condition 
 {\emph{there is at least one}} weakly hyperbolic sequence.
The goal of this paper is to recover results in 
 the spirit of \cite{Hu,Ed} in such a setting.
 
 \medskip
 
We briefly sketch our main results and philosophy of our approach, postponing the precise statements. As a general principle, rephrasing  Pugh-Shub principle \cite{PuSh}, we show that ``a little hyperbolicity goes a long way  guaranteeing stability-like properties''.
Here by a ``little hyperbolicity''  we understand either the almost-sure existence of weakly hyperbolic sequences or the existence of at least one, according to the case.
First, assuming that the set $S_{\mathrm{t}}$ has ``probability one'', we prove that the Markov operator is asymptotically stable (here we consider Markov measures associated to transition matrices and the particular case of Bernoulli probabilities). Second, 
we prove that if the Barnsley-Hutchinson operator has a unique fixed point then the IFS is asymptotically
stable.
 Finally,  in the case when $X$ is an interval, to establish the stability of the Markov operator we show that it is enough to assume that there are no common fixed points for the maps of the IFS and that there exists at least one weakly hyperbolic sequence.

 \medskip

 We now provide  more details for our main results (for the precise definitions and statements
see Section~\ref{s.precisestatent}).
 Associated to the set $S_{\mathrm{t}}$ of weakly hyperbolic sequences we consider
the \emph{coding map} $\pi\colon S_{\mathrm{t}}\to X$ that projects $S_{\mathrm{t}}$ into the phase space $X$,
see equation~\eqref{e.pi}.
The set $A_{\mathrm{t}}\eqdef\pi(S_{\mathrm{t}})$ is called the \emph{target} set and contains relevant dynamical information of the IFS.
Assuming that $S_{\mathrm{t}}\ne\emptyset$,  we prove the following results:
\begin{itemize}
\item The closure of the target set $\overline{A_{\mathrm{t}}}$ 
is a Conley attractor if and only if  it is a strict attractor (Theorem~\ref{mt.local}).
\item
The set $\overline{A_{\mathrm{t}}}$
is the global maximal fixed point of the IFS if and only if the
IFS is asymptotically stable. 
Moreover, the
Barnsley-Hutchinson operator has a unique fixed point if and only if
it is  asymptotically stable (Theorem~\ref{mt.att}).
\end{itemize}

We will investigate more closely  the  relation between 
target sets and  semifractals introduced in \cite{Lasota}.
An IFS $\mathfrak{F}=\mathrm{IFS}(T_{1},\dots, T_{k})$ 
 is said to be \emph{regular} if there
 are numbers $1\le i_{1}<i_{2}<\dots< i_{\ell} \leq k$ such that 
 $\mathfrak{F}'=\mathrm{IFS}(T_{i_{1}},\dots,T_{i_{\ell}})$ is asymptotically
stable. The global attractor of $\mathfrak{F}'$ is called a 
{\emph{nucleus}} of $\mathfrak{F}$ (an IFS may have several
nuclei).
By \cite{Lasota} for  any regular IFS $\mathfrak{F}$
there exists 
{\emph{the}} minimum fixed point of $\mathfrak{F}$, called its \emph{semifractal} and denoted
by $\mathrm{Semi}(\mathfrak{F})$.
It is obtained  from any nucleus of $\mathfrak{F}$ and 
attracts every compact set inside it, where iterations are taken 
with respect the Barnsley-Hutchinson operator of $\mathfrak{F}$. 
On the other hand, when $S_{\mathrm{t}}\ne \emptyset$,
the set $\overline{A_{\mathrm{t}}}$ is a minimum fixed point that
attracts every compact set inside it.
This provides the following characterisation of semifractals:
\begin{itemize}
\item
If an IFS $\mathfrak{F}$ is
regular
and satisfies $S_{\mathrm{t}}\neq \emptyset$ then
$\mathrm{Semi}(\mathfrak{F})=\overline{A_{\mathrm{t}}}$. 
\end{itemize} 
For a non-regular IFS with $S_{\mathrm{t}}\neq \emptyset$ (see 
 Example~\ref{ex.nonregular}) the
set $\overline{A_{\mathrm{t}}}$  plays the same role as a  semifractal plays for a regular IFS.
We refer to Remark~\ref{r.putting}
to support this assertion.

\medskip

We will also study the consequence of our approach for the so-called chaos game.
The {\emph{chaos game}} is an algorithm for generating fractals using 
random iterations of an IFS, see \cite{Ba}. It has probabilistic and 
disjunctive
(deterministic) versions, see \cite{Ba, Barnsleychaos,Pablo,Lesniak}.
Given an initial point $x=x_0\in X$, one considers the orbit 
$x_{n+1}= T_{\xi_n} (x_n)$, where 
the sequence  $\xi\in \Sigma_k^+$ is chosen according to
 some probability (\emph{probabilistic game})
or is a disjunctive sequence
(\emph{disjunctive game}). Recall that  $\xi\in \Sigma_k^+$ is {\emph{disjunctive}} if its  orbit
 (with respect to the usual left shift $\sigma$ defined by $\sigma(\xi)_n=\xi_{n+1}$) 
is dense in $\Sigma_k^+$.
The \emph{chaos game holds} when the sequence of \emph{tails}
$( \{ x_n \colon n\ge \ell \})_\ell$ in the Hausdorff distance
 converges to some attracting ``fractal'' (in such a case we also say that
 \emph{chaos game yields the fractal}). 
 
A natural question is how typically this game holds, where the term typical either  refers  to
sequences in $\Sigma_k^+$ or points in the phase space $X$.
By~\cite{Barnsleychaos}, the probabilistic chaos game holds 
when the fractal is a
  strict attractor and the initial point is in its basin of attraction. 
By~\cite{Pablo}, the disjunctive chaos game holds for a special class of 
 attractors\footnote{Called \emph{well-fibered} attractors, see also the strongly fibered case in \cite{Lesniak}.}
and every point in the pointwise basin of attraction.


In the context of the chaos game,
\cite{Lasota} considers IFSs
whose maps are Lipschitz with constants 
less than or equal to
$1$  and have
at least one uniformly contracting map. It is proved
that the probabilistic chaos game 
starting at
any point of the phase space 
yields the semifractal (even if the semifractal is not an attractor).
In our setting,
we get a similar result for the disjunctive chaos 
game where the fractal is the closure of the target set.

A fixed point $A$ of the Barnsley-Hutchinson operator
is  \emph{stable} if for every open neighbourhood $V$ of $A$ there
is an open  neighbourhood $V_{0}$ of $A$ such that 
\begin{equation}
\label{e.stableBH}
\mathcal{B}^{n}(V_{0})\subset V \quad \mbox{for every}\quad n\geq 0.
\end{equation}
For instance,
the set $\overline{A_{\mathrm{t}}}$ is stable when it is a Conley attractor 
or when 
all the  maps of the IFS  are Lipschitz with  constants less than or equal to $1$
(the existence of a contracting map is not required).
See Section \ref{ss.conleyandstrict} for 
an example where $\overline{A_{\mathrm{t}}}$ is stable 
but is not a Conley attractor.
\begin{itemize}
\item 
When  $\overline{A_{\mathrm{t}}}$ is a stable fixed point the 
{\emph{disjunctive chaos game}} holds for 
every point in the phase space (Theorem~\ref{mt.jogodocaos}).
\end{itemize}

%

\medskip

Finally we  consider IFSs from the ergodic point of view, studying the existence and uniqueness of stationary measures.
Recall that given an space of finite measures $\mathfrak{M}(X)$ defined on a set $X$, an operator
$\mathfrak{T}\colon  \mathfrak{M}(X)\to  \mathfrak{M}(X)$ 
such that
\begin{itemize}
\item
 $\mathfrak{T}$  is linear and
 \item
$\mathfrak{T}\nu (X)=\nu (X)$ for every $\nu \in  \mathfrak{M}(X)$
\end{itemize}
is called a {\emph{Markov operator.}}  A {\emph{stationary measure of $\mathfrak{T}$}}
is a fixed point of $\mathfrak{T}$. The operator  $\mathfrak{T}$ is
{\emph{asymptotically stable}} if it has a stationary measure $\nu$ 
such that $\lim \mathfrak{T}^n \mu = \nu$ for every $\mu \in \mathfrak{M}(X)$, in the weak$\ast$ topology.
The ergodic study of IFSs deals with two main   settings:
\begin{itemize}
\item 
{\emph{IFSs with probabilities}} given by a Bernoulli probability $\mathfrak{b}$ that  assigns
(positive)  
weights to  each map; 
\item
{\emph{Recurrent IFSs}} associated to an irreducible transition matrix $P$ inducing a Markov probability $\mathbb{P}^+$.
\end{itemize}
From the ergodic viewpoint one studies the iterations of points by an IFS (random orbits) 
as a Markov process and
each type of IFS has associated a special type of Markov operator
(associated to Bernoulli probabilities and associated to
transition matrices).
For a discussion see \cite{BarElton,RecurrentBar}.

When $S_{\mathrm{t}}\ne\emptyset$ and $X=[0,1]$
our ergodic results are summarised as follows:
\begin{itemize}
\item
Every injective IFS with Bernoulli probability  $\mathfrak{b}$
whose target set ${A_{\mathrm{t}}}$
is not a \emph{singleton} (i.e., 
 has at least two points)
is asymptotically stable and its 
unique stationary measure is
$\pi_{*}\mathfrak{b}$  
and satisfies
 $\mathrm{supp}(\pi_{*}\mathfrak{b})=\overline{A_{\mathrm{t}}}$. In this case, 
$ \overline{A_{\mathrm{t}}}$ is uncountable and the stationary measure is continuous (Theorem~\ref{mt.eimportante}). We will see that condition $\# (A_{\mathrm{t}})\ge 2$ 
($\#(A)$ means the cardinality of the set $A$)
implies that $\mfb (S_{\mathrm{t}})=1$.  For IFSs with $S_{\mathrm{t}}\ne \emptyset$ we see that if the Markov operator
associated to a Bernoulli probability $\mfb$ is asymptotically stable then the support of its stationary 
measure is $\overline{A_{\mathrm{t}}}$, even when $\mfb (S_{\mathrm{t}})=0$,
see Proposition~\ref{p.ultimosdias} (this proposition does not require $X=[0,1]$).
 \item
 An injective  recurrent IFS with a {\emph{splitting Markov measure}}\footnote{This is an ergodic version of the
 condition ``the set $A_{\mathrm{t}}$ is not a singleton'' and  means that there is 
 $i$ such that the restriction of $\pi$ to $[i]\cap \mathrm{supp}(\mathbb{P}^+)$
 is not constant.}
 $\mathbb{P}^+$ satisfies
$\mathbb{P}^{+}(S_{\mathrm{t}})=1$ (Theorem~\ref{mt.abundance}). We also get 
sufficient conditions for the asymptotically stability of a recurrent IFS and characterise its
unique stationary measure (Theorem~\ref{mt.generalizedB}).
\end{itemize}

\medskip

This paper is organised as follows.
In Section~\ref{s.precisestatent} we state  the main definitions and the precise statements of our results.
Section~\ref{s.attractors} is devoted to the study of different types of  attractors of IFSs and to the proofs of
Theorems~\ref{mt.local}, \ref{mt.att}, and 
\ref{mt.jogodocaos}. In Section~\ref{s.asstableifs01},
we consider
IFSs 
defined on the interval $[0,1]$, study the measure of   $S_{\mathrm{t}}$ for Markov measures,
and prove 
Theorem~\ref{mt.abundance}. 
We also get results about probabilistic rigidity of  $S_{\mathrm{t}}$
(Theorem~\ref{t.rigidity}) and characterise separable IFSs (Theorem~\ref{t.separableee}).
In Section~\ref{s.stationary} we prove Theorems~\ref{mt.eimportante} and 
\ref{mt.generalizedB} about stability of the Markov operator.
Finally, in Section~\ref{s.examples} we present some examples.

\section{Precise statement of results}
\label{s.precisestatent}

 \subsection{Topological properties of IFSs}\label{ss.ifstopological}
Consider the set 
 $S_{\mathrm{t}}$ of weakly hyperbolic sequences  
 in \eqref{e.stweakly}
 and define  the {\emph{coding map}}\footnote{This is the standard terminology for the map $\pi$ when
$S_{\mathrm{t}}=\Sigma_k^+$.}
\begin{equation}\label{e.pi}
\pi\colon S_\mathrm{t} \to X 
\quad\text{ by }\quad
\pi(\xi)\eqdef\lim_{n\to \infty}T_{\xi_{0}}\circ\cdots\circ T_{\xi_{n}}(p),
\end{equation}
where $p$ is any point of $X$.
By definition of the set $S_{\mathrm{t}}$, this limit always exists 
and is independent of $p\in X$. We introduce the {\emph{target set}} 
$A_{\mathrm{t}}\eqdef 
 \pi(S_{\mathrm{t}})$. This name is justified by the following characterisation
\begin{equation}\label{e.characterizationAt}
 A_{\mathrm{t}}=\{x\in X \colon \mbox{there is $\xi \in \Sigma_k^+$ with $\{x\}=\bigcap_n 
 T_{\xi_0}\circ\dots\circ T_{\xi_n}(X)$}\},
\end{equation}
see \eqref{e.At}.  
  The target set  plays a key role in the study of strict attractors.
We prove that if
$S_{\mathrm{t}}\neq \emptyset$ then the IFS has at 
most one strict attractor. Moreover, if such a strict 
attractor exists then 
it is equal to $\overline{A_{\mathrm{t}}}$, see Proposition~\ref{p.katia}.

 \begin{mteo}\label{mt.local}
 Consider an $\mathrm{IFS}$ defined on a compact metric space such that 
 ${S_{\mathrm{t}}}\ne \emptyset$.
 Then
  $\overline{A_{\mathrm{t}}}$ is a 
 Conley attractor if and only if it is a strict attractor. 
 \end{mteo}
%

In  \cite{Bar} Barnsley and Vince 
consider IFSs consisting either of affine maps or of M\"obius maps
and
introduce
sufficient conditions 
that guarantee the existence of a unique strict attractor.
 The proof involves some type of local  hyperbolicity 
 in a neighbourhood of a Conley attractor, see \cite{Atkins,Vince}.
  We point out that  Theorem~\ref{mt.local} only requires the existence of \emph{at least one} weakly hyperbolic sequence.


\medskip

Given an IFS $\mathfrak{F}$ and its
Barnsley-Hutchinson operator $\mathcal{B}_{\mathfrak{F}}$,
a subset $Y\subset X$ is \emph{$\mathcal{B}_{\mathfrak{F}}$-invariant} if $\mathcal{B}_{\mathfrak{F}}(Y)\subset Y$.  The closure of any $\mathcal{B}_{\mathfrak{F}}$-invariant set  
contains some fixed point of $\mathcal{B}_{\mathfrak{F}}$ (see the discussion below). Therefore,
since $X$ is $\mathcal{B}_{\mathfrak{F}}$-invariant, the operator $\mathcal{B}_{\mathfrak{F}}$ always 
has at least one fixed point. Indeed, we have a more precise description of the 
fixed points of $\mathcal{B}_{\mathfrak{F}}$. Following \cite{Ed},
given $Y\subset  X$ define  the set
\begin{equation}
\label{e.conjestrela}
Y^* \eqdef
\bigcap_{n\geq 0} \mathcal{B}_{\mathfrak{F}}^{n}(Y).
\end{equation}
If $Y$ is $\mathcal{B}_{\mathfrak{F}}$-invariant then
the set $(\overline{Y})^*$
is the {\emph{ global maximal fixed point of the restriction of $\mathcal{B}_{\mathfrak{F}}$ 
(or of the IFS) to 
the subsets of $\overline{Y}$}},  see Proposition~\ref{p.existence}.
The next theorem generalizes \cite{Ed} in two ways:  
it applies also to IFSs which are not weakly hyperbolic  and it provides a
necessary and sufficient condition for the existence of a global attractor.

\begin{mteo}\label{mt.att}
Consider an IFS ${\mathfrak{F}}$ defined on a compact metric space $X$ such that
$S_{\mathrm{t}}\ne \emptyset$.
Then the following three assertions are equivalent: 
\begin{enumerate}
 \item 
$\overline{A_\mathrm{t}}=X^*$,
 \item
the Barnsley-Hutchinson operator $\mathcal{B}_{\mathfrak{F}}$ has a unique fixed point,
\item
$X^{*}$ is a global attractor of the IFS ${\mathfrak{F}}$.
\end{enumerate}
\end{mteo}
 
We observe that the
statement in Theorem~\ref{mt.att} is sharp. Indeed, there are examples of non-weakly hyperbolic IFSs  where 
$A_{\mathrm{t}}\subsetneq \overline{A_{\mathrm{t}}}= X^*$, see Section~\ref{s.examples}.

Let us observe that for
 weakly hyperbolic IFSs it holds $A_{\mathrm{t}}=X^{*}$, see Lemma \ref{l.tiane}
 and also \cite{Ed}.
 We observe that  there are IFSs that are non-weakly hyperbolic such that
$A_{\mathrm{t}}=  \overline{A_{\mathrm{t}}}= X^{*}$, see Section~\ref{s.examples}.
%
%

\begin{mteo}[Disjunctive chaos game]
\label{mt.jogodocaos}
Consider  an $\mathrm{IFS}(T_{1},\dots,T_{k})$ defined on a compact metric space $X$
such that $\overline{A_{\mathrm{t}}}$ is a stable fixed point of the Barnsley-Hut\-chin\-son operator.
 Then for every  $x\in  X$ and every disjunctive sequence $\xi\in \Sigma_k^+$ we have
$$
\overline{A_{\mathrm{t}}}=\bigcap_{\ell \geq 0} \overline{\{x_{n,\xi}\colon n\geq \ell\}},
\quad\text{where}\quad
 x_{n,\xi}\eqdef T_{\xi_{n}}\circ\dots \circ T_{\xi_{0}}(x).
$$ 
In particular
 $$
\lim_{\ell\to \infty} \{x_{n,\xi}\colon n\geq \ell\}= \overline{A_{\mathrm{t}}},
$$
where
the limit is considered in the Hausdorff distance.
\end{mteo}

\subsection{Ergodic properties of IFSs}\label{ss.ifsprobabilities}

\subsubsection{IFSs with probabilities} \label{sss.ifsprobailities}
Consider an $\mathrm{IFS}(T_{1},\dots , T_{k})$ defined on a compact metric space $X$ and
strictly positive numbers 
$p_{1},\ldots,p_{k}$ 
(called {\emph{weights}})
such that $\sum_{i=1}^{k}p_{i}=1$.
We denote by
$\mathfrak{b}=\mathfrak{b} (p_1,\dots,p_k)$ the (non-trivial) Bernoulli probability measure
with weights $p_1, \dots,p_k$ defined on $\Sigma_k^+$.
We denote by $\mathrm{IFS}(T_{1},\dots T_{k};\mfb)$
the IFS with the corresponding Bernoulli probability and say that it is an {\emph{IFS with probabilities.}} 

Let $\mathcal{M}_{1} (X)$ be the space of  
Borel probability measures defined on $X$ equipped with the weak$*$-topology. 
The \emph{Markov operator}  associated to the $\mathrm{IFS}(T_{1},\dots T_{k};\mfb)$ is defined by
\begin{equation}
\label{e.markov}
\mathfrak{T}_\mfb \colon \mathcal{M}_{1}(X)\rightarrow \mathcal{M}_{1}(X),
\qquad
\mathfrak{T}_\mfb \mu\eqdef \sum_{i=1}^{k}p_{i} \, T_{i\ast}\mu,
\end{equation}
where $T_{i\ast}\mu (A)= \mu (T_i^{-1}(A))$ for every Borel set $A$.
Note that the Markov operator $\mathfrak{T}_\mathfrak{b}$ is continuous. Hence, if
$\mfT_\mathfrak{b}$
is asymptotically stable
then its  attracting measure $\mu$ is stationary, that is, satisfies $\mfT_\mathfrak{b}\mu=\mu$.

An IFS with probabilities  $\mathrm{IFS}(T_{1},\dots T_{k};\mfb)$  is called
{\emph{asymptotically stable}} if its Markov operator $\mathfrak{T}_\mfb$ is asymptotically 
stable.
It is a folklore result that if $\mathfrak{b}(S_{\mathrm{t}})=1$ then the IFS is
asymptotically stable and $\pi_{*}\mathfrak{b}$ is the unique stationary measure, see for instance 
\cite{Sten, Letac}.
In Proposition \ref{p.attractingstationary} we prove
this fact and we see that
 $\mathrm{supp}(\pi_{*}\mathfrak{b})=\overline{A_{\mathrm{t}}}$.
 Note that, since that
 $\sigma^{-1}(S_{\mathrm{t}})\subset S_{\mathrm{t}}$,
 the ergodicity 
 of the Bernoulli  measure (with positive weights) $\mfb$ with respect to the shift 
 implies that either $\mathfrak{b}(S_{\mathrm{t}})=1$ or $\mathfrak{b}(S_{\mathrm{t}})=0$.

%
%
%
%
%

A combination of Theorem~\ref{mt.att} and Proposition~\ref{p.attractingstationary} 
allows us to recover properties of hyperbolic IFSs in non-hyperbolic settings provided that
 the
sets 
${A_{\mathrm{t}}}$ and 
$S_{\mathrm{t}}$ are ``big enough'' (from the topological and probabilistic points of view, respectively): 
there are a unique  global  attractor  and the IFS with probabilities is
asymptotically stable.

 Proposition~\ref{p.attractingstationary}  assumes that $\mathfrak{b}(S_{\mathrm{t}})=1$
(which is often difficult to verify).
When $X=[0,1]$  we  improve   this proposition replacing the condition 
$\mathfrak{b}(S_{\mathrm{t}})=1$ by the topological
  condition $\#(A_{\mathrm{t}})\ge 2$
     that we call \emph{separability} and it is quite straightforward to verify.


\begin{mteo}\label{mt.eimportante}
Consider an $\mathrm{IFS}(T_{1},\dots T_{k})$
defined on $[0,1]$ such that
\begin{itemize}
\item
the target set  $A_{\mathrm{t}}$  has at least two elements
and
\item
there is a non-trivial closed interval $J\subset [0,1]$ such that
 $T_i(J)\subset J$ and $T_i|_{J}$ is injective for every $j\in \{1,\dots,k\}$.
 \end{itemize}
 Then for every (non-trivial) Bernoulli probability $\mathfrak{b}$ the  $\mathrm{IFS}(T_{1},\dots T_{k};\mathfrak{b})$ 
is asymptotically stable.

Moreover,  $\pi_{*}\mathfrak{b}$ is the
(unique)  stationary measure of  
$\mathrm{IFS}(T_{1},\dots T_{k};\mathfrak{b})$, satisfies
 $\mathrm{supp}(\pi_{*}\mathfrak{b})=\overline{A_{\mathrm{t}}}$, and is continuous.
As a consequence, the set $A_{\mathrm{t}}$ 
 has no isolated points.
%
\end{mteo}

In the previous theorem, the purely topological condition $\#(A_{\mathrm{t}})\ge 2$
depending only on  $\mathrm{IFS}(T_{1},\dots T_{k})$ implies the asymptotic stability of
the Markov operator $\mathfrak{T}_\mathfrak{b}$ of $\mathrm{IFS}(T_{1},\dots T_{k};\mathfrak{b})$
for any (non-trivial) Bernoulli probability $\mathfrak{b}$. Moreover, we also obtain properties of the stationary measure. The support of this stationary measure is independent of the Bernoulli probability.
In Proposition~\ref{p.ultimosdias} we state a result 
about the support of stationary measures
that holds for general compact metric spaces: if
$S_{\mathrm{t}}\ne \emptyset$ and the Markov operator associated 
to $\mfb$ is asymptotically stable then the support of its stationary 
measure always is $\overline{A_{\mathrm{t}}}$, even when $\mfb (S_{\mathrm{t}})=0$.

The asymptotic stability of an IFS with probabilities has been obtained in
several contexts such as, for example, \emph{contracting on average} \cite{BarElton},
weakly hyperbolic \cite{Ed}, and \emph{non-overlapping}\footnote{
An IFS  is called non-overlapping if the maps $T_{i}$ are injective and the sets $T_{i}(I)$ 
have disjoint interiors.  
We will see that separability is a weak form of non-overlapping, see 
Theorem~\ref{t.separableee}. We observe that \cite{Ed} and \cite{BarElton} do not involve injective-like conditions of the IFS.}
 \cite{Stri}. Observe that the contexts of \cite{BarElton,Ed} have a hyperbolic flavour.
  Let us also observe that \cite{SzZd} states the asymptotic stability of admissible IFSs 
   consisting of circle homeomorphisms (these homeomorphisms preserve  the orientation and 
   some homeomorphism of the IFS is transitive). Note that in this case the set $S_{\mathrm{t}}$ is empty.
Let us compare these results with Theorem~\ref{mt.eimportante}. 
First, the condition to be contracting on average depends on the selected Bernoulli probability
(an IFS may be contracting in average with respect to some probabilities but not with respect
to  {\emph{all}}
Bernoulli probabilities).
In contrast, weak hyperbolicity, separability, non-overlapping, and admissibility conditions are 
topological conditions that do not involve
 probabilities. These conditions guarantee the
asymptotic stability of the
Markov operator $\mathfrak{T}_\mathfrak{b}$ of the IFS with respect to \emph{any} Bernoulli probability
$\mathfrak{b}$.


Finally, note that checking the properties of weak hyperbolicity and contracting on average may be rather complicated,
while the separability condition is comparably much  simpler, thus
 Theorem \ref{mt.eimportante} can also be useful in these contexts. 

\subsubsection{Recurrent IFSs}\label{sss.recurrent}
A generalization of IFSs with probabilities are the so-called \emph{recurrent IFSs} introduced 
 in \cite{RecurrentBar}, where the 
weights $p_{i}$ are replaced by a transition matrix.

To be more precise, recall that
a $k\times k$ matrix  $P=(p_{ij})$ is a 
\emph{transition matrix} if 
$p_{ij}\ge 0$ for all $i,j$ and 
for every $i$ it holds
$\sum_{j=1}^{k}p_{ij}=1$.
An \emph{stationary probability vector} associated to  $P$  is a 
  vector $\bar p=(p_{1},\ldots,p_{k})$ whose elements are 
  non-negative real numbers and sum up to $1$ and satisfies $\bar p \, P= \bar p$.
The transition matrix $P$ is called \emph{irreducible} if  
for every $\ell,r\in \{1,\dots,k\}$ there is $n=n(\ell,r)$ such that
$P^{n}=(p^n_{ij})$ satisfies $p^n_{\ell,r}>0$.
An irreducible transition matrix has a unique   stationary
probability vector $\bar p=(p_{i})$, see 
\cite[page 100]{Keme}.
We consider
the {\emph{cylinders}}
 $$
 [a_{0}\dots a_{\ell}]\eqdef \{\omega\in \Sigma_k^+ \colon \omega_{0}=a_{0},\dots,\omega_{\ell}=a_{\ell}\}
 \subset\Sigma_k^+
 $$ 
which is a semi-algebra that generates the Borel $\sigma$-algebra of $\Sigma_k^+$.
We denote by
$\mathbb{P}^{+}$
the {\emph{Markov measure}} associated to $(P,\bar p)$ defined  on $\Sigma_{k}^{+}$,
this measure is defined on the cylinders $[a_{0}\ldots a_\ell]$  by
 $$
 \mathbb{P}^+([a_{0}\ldots a_\ell])\eqdef
 p_{a_{0}}p_{a_{0}a_{1}}\ldots p_{a_{\ell-1}a_\ell}.
 $$ 

Given an $\mathrm{IFS}(T_{1},\dots T_{k})$ 
and an irreducible transition matrix  $P=(p_{ij})$,
we call
 $\mathrm{IFS}(T_{1},\dots T_{k};\mathbb{P}^{+})$ a \emph{recurrent $\mathrm{IFS}$}. We now  introduce the Markov operator in 
 this context.
 Consider the set $\widehat{X}\eqdef X\times \{1,\dots, k\}$ with the 
 product topology and the corresponding Borel sets. 
Given a subset $\widehat{B}\subset\widehat{X}$, its 
{\emph{$i$-section}} is defined by
$$
\widehat B_{i}\eqdef\{x\in X\colon (x,i)\in \widehat{B}\}. 
$$
The {\emph{$i$-section of a probability 
measure $\widehat{\mu}$}} on $\widehat{X}$ is defined on the set $X$ by
$$
\mu_{i}(B)\eqdef \widehat{\mu}(B\times \{i\}), 
\quad \mbox{
where $B$ is any Borel subset of $X$.}
$$
Observe that 
${\mu}_i$ is a finite measure on $X$ but, in general, it is 
not a probability measure. 
Since the measure $\widehat \mu$ is completely defined by its sections
we write $\widehat{\mu}=(\mu_{1},\dots,\mu_{k})$
and note that   
$$
\widehat{\mu}(\widehat B)=\sum_{j=1}^{k}\mu_{j}(\widehat B_{j})
\quad \mbox{for every Borel subset $\widehat B$
of $\widehat{X}$}.
$$
%

The
\emph{(generalised) Markov operator}
of recurrent $\mathrm{IFS}(T_{1},\dots T_{k};\mathbb{P}^{+})$
 is defined by  
  \begin{equation}\label{e.generalised}
   \mathfrak{S}_{\mathbb{P}^{+}}\colon \mathcal{M}_{1}(\widehat{X})\to\mathcal{M}_{1}(\widehat{X}), \quad
   \widehat \mu \mapsto \mathfrak{S}_{\mathbb{P}^{+}}(\widehat{\mu}),
  \end{equation}
 where
 $$
  \mathfrak{S}_{\mathbb{P}^{+}}(\widehat{\mu})(\widehat{B})\eqdef \sum_{i,j}p_{ij}T_{j*}\mu_{i}(\widehat B_{j}).
$$
   

A recurrent $\mathrm{IFS}(T_{1},\dots T_{k};\mathbb{P}^+)$  is called
{\emph{asymptotically stable}} if 
the Markov operator 
$\mathfrak{S}_{\mathbb{P}^+}$ 
is asymptotically  stable. 

Given a Markov measure $\mathbb{P}^+$ there is associated its 
{\emph{inverse Markov measure}} 
$\mathbb{P}^-$ defined on $\Sigma_k^+$ by
\begin{equation}\label{e.markovinverse}
\mathbb{P}^- ([a_{0}a_1\dots a_n])\eqdef
\mathbb{P}^+ ([a_{n}\dots a_{1}a_{0}]),
\quad
\mbox{for a cylinder $[a_{0}a_1\dots a_n]$.}
\end{equation}
 The measure $\mathbb{P}^-$  is also Markov (see Section~\ref{sss.inversemarkov}).

There is the following
 \emph{generalised coding map} from $S_{\mathrm{t}}$ to $\widehat{X}$ defined by
\begin{equation}\label{e.generalisedcodingmap}
\varpi \colon S_{\mathrm{t}} \to \widehat{X}, \qquad \varpi(\xi)\eqdef(\pi(\xi),\xi_{0}).  
\end{equation}

In Theorem~\ref{t.p.t.recurrent} we see that if a recurrent
$\mathrm{IFS}(T_{1},\dots T_{k};\mathbb{P}^{+})$  is
such that $\mathbb{P}^{-}(S_{\mathrm{t}})=1$ 
and $\mathbb{P}^{-}$ is mixing then it
is asymptotically stable 
and the stationary measure of $\mathfrak{S}_{\mathbb{P}^+}$ is
 $\varpi_{*}\mathbb{P}^{-}$, that is,
 $$
 \mathfrak{S}^n_{\mathbb{P}^+} (\widehat \mu) 
 \underset{*}{\to} \varpi_{*}\mathbb{P}^{-}
 \quad
 \mbox{for every $\widehat{\mu}\in \mathcal{M}_1 (\widehat X)$}.
 $$
 This is a
version of Proposition~\ref{p.attractingstationary} for recurrent IFSs.

As in the case of IFSs with probabilities,
when $X=[0,1]$  we can improve 
Theorem~\ref{t.p.t.recurrent}. In this proposition it is assumed that
$\mathbb{P}^-(S_{\mathrm{t}})=1$
(verifying this assumption is in general difficult). When $X=[0,1]$ we
can replace this condition by a ``splitting
 condition''   that is quite straightforward to verify.

Consider a
recurrent $\mathrm{IFS}(T_{1},\dots T_{k};\mathbb{P}^{+})$
defined on a compact metric space $X$.
A cylinder $[j_{1}\dots j_{s}]$ is called {\emph{admissible}} if 
$\mathbb{P}^{+}([j_{1}\dots j_{s}])>0$.

\begin{defi}[Splitting Markov measure]
\label{d.splits}
Consider an IFS
$\mathfrak{F}=\mathrm{IFS}(T_1,\dots,T_k)$ defined on $[0,1]$ and a non-trivial closed
 interval $J$ of $[0,1]$.
A Markov measure $\mathbb{P}^{+}$ defined on $\Sigma_{k}^{+}$
\emph{splits the IFS $\mathfrak{F}$ in $J$} if
 \begin{itemize}
 \item
 $T_i(J)\subset J$ and $T_i|_{J}$ is injective for every $j\in \{1,\dots,k\}$,
 \item
  there are admissible cylinders
$[i_{1}\dots i_{\ell}]$  and
 $[j_{1}\dots j_{s}]$ of $\mathbb{P}^+$ 
 with  $i_{1}=j_{1}$
 such that 
  $$
 T_{j_{1}}\circ\dots\circ T_{j_{s}}(I)\cap T_{i_{1}}\circ\dots \circ T_{i_{\ell}}(I)=\emptyset
 $$ 
 and 
 $$
  T_{j_{1}}\circ\dots\circ T_{j_{s}}(I) \cup T_{i_{1}}\circ\dots \circ T_{i_{\ell}}(I)
  \subset J.
 $$
 \end{itemize}
 When $J=I$ we say that  $\mathbb{P}^+$ 
{\emph{splits $\mathfrak{F}$}}.
\end{defi}
%
%

Let $T$ be a measure-preserving transformation on 
 a probability space $(X,\mathbb{B},\mu)$.  
 Recall that  $(T,\mu)$ is  \emph{ergodic} if for 
 every measurable set $A$ with 
$T^{-1}(A)=A$ it holds $\mu(A)=0$ or $\mu(A)=1$. 
Recall that $(T,\mu)$ is \emph{mixing} if 
$$
\lim_{n\to \infty} \mu(T^{-n}(A)\cap B)=\mu(A)\, \mu(B)
\quad
\mbox{for every } A,B\in \mathbb{B}.
$$
A Borel measure $\mu$ on $\Sigma_{k}^{+}$ is  \emph{mixing} if the system
$(\sigma,\mu)$ is mixing.

Next theorem states consequences of the splitting property of a Markov measure and is the
main tool to get  the asymptotic stability of the  Markov operator.

\begin{mteo}\label{mt.abundance}
Consider an $\mathrm{IFS}(T_{1},\dots T_{k})$
defined on the interval $[0,1]$. If $\mathbb{P}^{+}$ is mixing Markov measure
that
splits the $\mathrm{IFS}$ in some non-trivial closed interval $J$ then $\mathbb{P}^{+}(S_{\mathrm{t}})=1$.
\end{mteo}

Next theorem gives sufficient conditions for the asymptotic stability of the Mar\-kov operator.
 
\begin{mteo}\label{mt.generalizedB}
Consider a recurrent $\mathrm{IFS}(T_{1},\dots T_{k};\mathbb{P}^{+})$ 
defined on the interval $[0,1]$. Suppose that
the 
inverse Markov measure $\mathbb{P}^{-}$ is  mixing and splits the $\mathrm{IFS}$
in some non-trivial closed interval $J$. 
Then $\mathrm{IFS}(T_{1},\dots T_{k};\mathbb{P}^{+})$ 
is asymptotically stable and $\varpi_{*}\mathbb{P}^{-}$ is the stationary measure 
of  the Markov operator $\mathfrak{S}_{\mathbb{P}^{+}}$.
 \end{mteo}

Note that $\mathbb{P}^+$ is mixing if and only if  $\mathbb{P}^-$ is mixing. However,  
a splitting property  for   $\mathbb{P}^+$ does not imply a splitting property for $\mathbb{P}^-$ (and 
vice-versa).

\subsection{Preliminaries and notation}\label{s.pre}
We now establish some basic definitions and notations.

\subsubsection{Distances}
Throughout this paper
$(X,d)$ is a compact metric space and $\mathcal{P}(X)$ denotes
 the power set of $X$.
 Given a point $x\in X$ and a set $A\subset X$,  distance between $x$ and $A$
 is defined
  by 
 $$
 d(x,A)\eqdef \inf\{d(x,a)\colon a\in A\}.
 $$
 The {\emph{Hausdorff distance}} between two sets $A,B\subset X$ is defined by 
$$
 d_{H}(A,B)\eqdef \max\{h_{s}(A,B),h_{s}(B,A)\},
 \quad
 \mbox{where} 
\quad
 h_{s}(A,B)\eqdef \sup_{a\in A}d(a,B).
 $$
 Note that, in general, $d_H$ is only a pseudo-metric  defined on $\mathcal{P}(X)$.
 Let $\mathcal{H}(X)\subset \mathcal{P}(X)$ be the set of all non-empty compact subsets of $X$.
Then $(\mathcal{H}(X),d_{H})$ is a compact metric space, see \cite{Ba}.
%
%
%
%
%
 
  \subsubsection{Inverse Markov measures}\label{sss.inversemarkov}
%

Consider a transition matrix $P=(p_{ij})$ and a stationary probability vector $\bar p=
(p_1, \dots, p_k)$ of $P$.
If all entries of $\bar p$ are (strictly) positive then the 
\emph{inverse transition matrix} associated to $(P,\bar p)$ is the matrix ${Q}_{(P,\bar p)}=(q_{ij})$ where
$$
q_{ij}\eqdef \frac{p_{j}}{p_{i}}\, p_{ji}.
$$
Note that $Q={Q}_{(P,\bar p)}$ is a transition matrix and $\bar p$ is a stationary probability vector of 
${Q}_{(P,\bar p)}$.  We observe that if $P$ is primitive if and only if $Q$ is primitive.

 Denote by $\mathbb{P}^{-}$ the Markov 
measure associated to $(Q,\bar p)$.  For every cylinder $[a_{0}\dots a_{\ell}]$ it holds 
$$
\mathbb{P}^{-}([a_{0}\dots a_{\ell}])=\mathbb{P}^{+}([a_{\ell}\dots a_{0}]),
$$
where $\mathbb{P}^+$ is the Markov measure associated to $(P,\bar p)$.

Let us observe 
that a Markov measure $\mathbb{P}^{+}$
is mixing if and only if the transition matrix is 
\emph{primitive}\footnote{Also called {\emph{aperiodic.}}} (i.e. there is $n\geq 1$ such that
all the entries of $P^{n}$ are strictly positive), see for instance \cite[page 79]{Brin}.
As a consequence, $\mathbb{P}^-$ is mixing if and only if $P$ is primitive.

\section{Attractors of iterated function systems}
\label{s.attractors}


 This section is devoted to the study 
of
fixed points and the attractors of the  Barnsley-Hutchinson operator of an IFS (see
Sections~\ref{ss.Bfixed} and \ref{ss.conleyandstrict}).
Our 
goal is to prove  Theorems~\ref{mt.local}, \ref{mt.att},
and \ref{mt.jogodocaos} (see Sections~\ref{ss.mtlocal}, \ref{ss.attractor}, and \ref{ss.mtjogo},
respectively).
We also get some topological properties of the target set $A_{\mathrm{t}}$ in Section~\ref{ss.structure}.

In what follows we consider $\mathfrak{F}=\mathrm{IFS}(T_{1},\dots T_{k})$
and denote by $\mathcal{B}_\mathfrak{F}=\mathcal{B}$ its 
 Barnsley-Hutchinson operator, recall \eqref{e.BH}.

\subsection{Fixed points for the Barnsley-Hutchinson operator}
 \label{ss.Bfixed}
 We will show that 
every compact invariant set $A$ of $X$ (i.e., $\mathcal{B}(A)\subset A)$
contains some fixed point of $\mathcal{B}$.
Since $X$ is invariant for $\mathcal{B}$ this implies that $\mathcal{B}$ always 
has at least one fixed point. 
To each set $A$ we associate the set  
$A^* \eqdef
\bigcap_{n\geq 0} \mathcal{B}^{n}(A)$, recall \eqref{e.conjestrela}.

Recall that $\mathcal{H}(X)$ denotes the set consisting of all non-empty compact subsets of $X$.
We consider in $\mathcal{H}(X)$ the Hausdorff distance $d_H$.

\begin{prop}[Existence of fixed points of $\mcalB$]\label{p.existence}
Consider  $A\in \mcalH (X)$ such that  $\mathcal{B}(A)\subset A$. 
Then $A^*$ is a  fixed point of $\mathcal{B}$. 
In particular, $X^{*}$ is a fixed point of $\mcalB$.
\end{prop}

\begin{proof}
The proposition follows from the
next lemma and the continuity of $\mathcal{B}$.

 \begin{lema}
 \label{l.semnome}
 Let $(A_{n})$ be a sequence of nested compact sets, $A_{n+1}\subset A_n$,
 and $A=\bigcap_{n\geq0} A_{n}$. Then
 $d_{H}(A_{n},A)\rightarrow 0$.
 \end{lema}
 
  \begin{proof}
The proof is by contradiction. If the lemma is false there are
 $\epsilon>0$ and a subsequence $(n_{\ell})$, $n_\ell \to \infty$, such that 
 $d_{H}(A_{n_{\ell}},A)\geq \epsilon$ for all $\ell$. 
 Since $A\subset A_{n}$, 
 for each $\ell$
 there is a point $p_{\ell}\in A_{n_{\ell}}$ such that $d(p_{\ell},A)\geq \epsilon$. 
 By compactness, taking a subsequence if necessary, we can assume that 
 $p_{\ell}\to p$. 
 As $ (A_{n})$ is nested 
 it follows that $p\in A$, contradicting  that
 $d(p_{\ell},A)\geq \epsilon$ for all $\ell$. 
\end{proof}
To prove the proposition it is enough to apply the lemma to nested sequence $A_n=\mathcal{B}^n(A)$.
\end{proof}

 Now let us look more closely to the fixed
point $X^{*}$ of $\mathcal{B}$.
For that to each  $\xi\in \Sigma_k^+$ we consider its 
{\emph{fibre}} defined by 
\begin{equation}
\label{spine}
I_{\xi}\eqdef 
\bigcap_{n\geq 0} T_{\xi_{0}}\circ\cdots\circ T_{\xi_{n}}(X), \quad
\mbox{if $\xi=\xi_0\xi_1\dots$}.
\end{equation}
We will see in Lemma~\ref{l.tiane}
that the set 
$X^{*}$ is the union of the fibres  $I_{\xi}$.

Note that every fibre is a non-empty set: just note
 that  $(T_{\xi_{0}}\circ\cdots\circ T_{\xi_{n}}(X))_{n\in \mathbb{N}}$ is a
sequence of nested compacts sets. Moreover, when  $X$ is
an interval, the fibres also are intervals (may be trivial ones). With this
definition, the set $ S_\mathrm{t}$  of weakly hyperbolic sequences, recall \eqref{e.stweakly}, is given by
$$
 S_\mathrm{t}=\{ \xi\in \Si_k^+ \colon \mbox{$I_\xi$ is a singleton}\}.
$$
 From the definition of the target set  $ A_{\mathrm{t}}$ in \eqref{e.characterizationAt} it immediately follows that
  \begin{equation}
  \label{e.At}
 A_{\mathrm{t}}= \bigcup_{\xi \in S_\mathrm{t}} I_\xi.
 \end{equation}
 Recall that by definition for every set $A$ we have
 \begin{equation}
  \label{e.Xestrela}
 A^{*}=\bigcap_{n\geq 0} \mathcal{B}^{n}(A)=\bigcap_{n\geq 0} \,\,\bigcup_{\xi\in\Sigma_{k}^{+}} T_{\xi_{0}}\circ\cdots\circ T_{\xi_{n-1}}(A).
\end{equation}
 Next lemma just says that the 
 operations $``\cup"$ and $``\cap"$  above commute.
 

\begin{lema}\label{l.tiane}
Let $A\in \mathcal{H}(X)$ such that $\mathcal{B}(A)\subset A$. Then 
$$
A^*=
\bigcap_{n\geq 0} \mathcal{B}^{n}(A)=\bigcup_{\xi\in\Sigma_{k}^{+}} 
\,\, \bigcap_{n\geq 0} T_{\xi_{0}}\circ\cdots\circ T_{\xi_{n}}(A). 
$$
In particular,   $$
 X^{*}=\bigcup_{\xi\in \Sigma_{k}^{+}} I_{\xi}.
 $$
\end{lema}

\begin{proof}
Condition  $\mathcal{B}(A)\subset A$ implies that 
 $\mathcal{B}^n(A)$ is a decreasing nested family of compact subsets and
 $T_i (A) \subset A$ for all $i=1,\dots, k$.
 From equation \eqref{e.Xestrela} it follows immediately that
$$
\bigcup_{\xi\in\Sigma_{k}^{+}} \, \bigcap_{n\geq 0} 
T_{\xi_{0}}\circ\cdots\circ T_{\xi_{n}}(A)\subset 
\bigcap_{n\geq 0}\,\, \bigcup_{\xi\in\Sigma_{k}^{+}} T_{\xi_{0}}\circ\cdots\circ T_{\xi_{n}}(A)=A^*.
$$
which implies the inclusion ``$\supset$''. 

To prove the inclusion ``$\subset$''
we use the following classical combinatorial lemma, see for instance
\cite[page 302]{Konig}. 

\begin{lema}[K\"onig's lemma]\label{l.konig}
Let $G$ be a connected graph with infinitely many vertices such that every vertex has 
finite degree. Then $G$ contains an infinite path with no 
repeated vertices.
\end{lema}

Take now any point 
$p\in \bigcap_{n\ge 0}\, \mcalB^n(A)$.
Then for each $n\geq 1$ there is a finite sequence 
$\beta^n_{0}\ldots\beta^n_{n-1}$ 
such that $p\in T_{\beta^n_{0}}\circ\cdots\circ T_{\beta^n_{n-1}}(A)$. 
As $\mathcal{B}(A)\subset A$ this implies that $p\in T_{\beta^n_{0}}\circ\cdots\circ T_{\beta^n_{\ell}}(A)$ for all 
$\ell\leq n-1$. 

We now apply 
Lemma~\ref{l.konig}
to the graph $G$ whose vertices are the sets
$$
\bigcup_{n\ge 0}
\left\{ T_{\beta^n_{0}}\circ\cdots\circ T_{\beta^n_{n-1}}(A),
T_{\beta^n_{0}}\circ\cdots\circ T_{\beta^n_{n-2}}(A),\cdots,
T_{\beta^n_{0}}(A),A\right\}.
$$
Note that, in principle, an vertex can be obtained  using different compositions.

The edges of the graph are defined as follows: the vertex 
$T_{\beta^n_{0}}\circ\cdots\circ T_{\beta^n_{\ell}}(A)$
 has edges joining to
$T_{\beta^n_{0}}\circ\cdots\circ T_{\beta^n_{\ell-1}}(A)$ and $T_{\beta^n_{0}}\circ\cdots\circ T_{\beta^n_{\ell+1}}(A)$
(provided $\ell-1\ge 0$ and $\ell+1\le n$). 

Observe that the way we define the graph $G$ allows that a pair of adjacent vertices may have 
infinite links a, thus in such a case the graph has not finite degree. To
bypass this difficulty,  
we consider the \emph{underlying simple} graph $G_{0}$ of $G$
 obtained by deleting from every pair of adjacent vertices all but one edge joining them.
For the underlying simple graph $G_0$ the ``top vertex'' $A$ has at most $k$ edges (joining to the sets $T_1(A), \dots , T_{k}(A))$ and the other vertices has at most $k+1$ edges.
 In this way, the graph $G_{0}$ has finite degree at most $k+1$.

Lemma~\ref{l.konig} now gives a simple path with infinite length. This simple path provides
 a sequence 
$\xi=\xi_{0}\ldots \xi_{n}\ldots$ such that 
$p\in T_{\xi_{0}}\circ\cdots\circ T_{\xi_{n}}(A)$ for every $n$. This implies the 
inclusion ``$\subset$''. 
\end{proof}

 \subsection{Conley and strict attractors}
 \label{ss.conleyandstrict}
 In this section we introduce the notion of a minimum fixed point of an IFS
 and prove that if  $S_{\mathrm{t}}\ne \emptyset$ then the closure of the target set
 is a minimum fixed point of $\mathcal{B}$. We also characterise strict attractors for IFSs with 
 $S_{\mathrm{t}}\ne \emptyset$.

 \subsubsection{Minimal fixed points and minimum fixed point}
 \label{minimalandminimum} 
Note that the set $X^{*}$ is the \emph{maximum fixed point} (ordered by inclusion)
 of the map $\mathcal{B}$, meaning that
if $K$ is another fixed point of $\mathcal{B}$ then $K\subset X^{*}$.
A natural  question is about the existence of a 
\emph{minimum fixed point} $Y$ of $\mathcal{B}$,  
meaning that
if $K$ is any fixed point of $\mathcal{B}$ then $Y\subset K$.
By definition,  maximum and minimum fixed points are  unique.
We see that, in general, may no exist a minimum fixed point.
 Observe  that
an application of  Zorn's lemma immediately provides a
\emph{minimal fixed point} for $\mathcal{B}$, that is,
a fixed point that does not contain properly another fixed point.
Note that, by definition, a minimum fixed point is minimal, but the converse is not true in general.

To get a simple example of an IFS without a minimum fixed point
just consider
the $\mathrm{IFS}(T_{1},T_{2})$ defined on the  interval $[0,1]$ with $T_{1}(x)=x$
and $T_{2}(x)=1-x$. For each $x\in [0,1]$, the set $\{x,1-x\}$ is a fixed 
point of $\mathcal{B}$. Clearly, the set $\{x,1-x\}$ is minimal.
It is also obvious, that there is not a 
fixed point contained in all fixed points. 
Thus  the minimal fixed points cannot be minimum fixed points.

In the previous example  we have $S_{\mathrm{t}}= \emptyset$.
Next proposition shows that the condition $S_{\mathrm{t}}\neq \emptyset$ 
guarantees the existence of a minimum fixed point. For the next result recall the characterisation
of the set $A_t$ in \eqref{e.At}.

\begin{prop}
\label{p.l.minimality}
Suppose that $S_\mathrm{t}\ne \emptyset$. Then $\overline{A_\mathrm{t}}$ is the minimum fixed point of $\mathcal B$.
\end{prop}

 \begin{proof}
 We need to see that
 $\mathcal B(\overline{A_\mathrm{t}})=\overline{A_\mathrm{t}}$ 
 and  $\overline{A_\mathrm{t}}\subset K$ 
 for every  compact set $K$ with $\mathcal B(K)=K$.

To prove the second assertion,  fix any compact set $K$ that is  fixed point of $\mathcal{B}$ and take any point $p\in A_\mathrm{t}$. By the 
 characterisation of $A_t$ in \eqref{e.At} 
 there is a sequence $\xi$ such that 
 \begin{equation}\label{e.asinequation}
 \{p\}=\bigcap_{n\geq 0} T_{\xi_{0}}\circ\cdots\circ T_{\xi_{n}}(X)
 \supset \bigcap_{n\geq 0} T_{\xi_{0}}\circ\cdots\circ T_{\xi_{n}}(K).
 \end{equation}
 Since the last intersection is non-empty and contained in $K$ it follows $p\in K$. 
 This implies that $A_\mathrm{t}$ (and hence $\overline{A_\mathrm{t}}$) 
 is contained in $K$.

To see that $\mathcal B(\overline{A_\mathrm{t}})=\overline{A_\mathrm{t}}$ note that 
the continuity of the maps $T_i$ 
implies that for $p$ as in \eqref{e.asinequation} and  every $i=1,\dots,k$ it holds
$$
\{T_{i}(p)\}=
\bigcap_{n\geq 0} T_{i}\circ T_{\xi_{0}}\circ\cdots\circ T_{\xi_{n}}(X).
$$
This implies that $\mathcal{B}({A_{t}})\subset{A_{t}}$.  Hence, by continuity of the maps $T_i$,
$\mathcal{B}(\overline{A_{t}})\subset\overline{A_{t}}$. By 
definition this implies that $(\overline{A_{\mathrm{t}}})^*\subset (\overline{A_{\mathrm{t}}})$. 
By Proposition ~\ref{p.existence}, $(\overline{A_{\mathrm{t}}})^*$ is a fixed point of $\mathcal{B}$. 
The minimality property proved before implies that $\overline A_{\mathrm{t}}\subset (\overline{A_{\mathrm{t}}})^*$.
This ends the proof of the proposition.
 \end{proof}
 
 Note that in the proof of the proposition we obtained the following.
 \begin{schol}\label{sc.Atinv}
 Given an $\mathrm{IFS}(T_{1},\dots T_{k})$ with $S_{\mathrm{t}}\ne \emptyset$ it holds
 $T_i(A_t)\subset A_t$.
 \end{schol}

\subsubsection{Characterisation of strict attractors}
\label{sss.characestrict} 
The proposition below claims that an IFS with a weakly hyperbolic sequence
has at most one  strict attractor and describes such an attractor.

\begin{prop}\label{p.katia}
Consider an $\mathrm{IFS}$ defined on a compact metric space such that 
$S_{\mathrm{t}}\neq \emptyset$. Then there exists at most one strict 
attractor. If such a strict attractor exists then it is equal to $\overline{A_{\mathrm{t}}}$.  
\end{prop}

\begin{proof} 
If there are no strict attractor we are done. Otherwise assume that there is 
a strict attractor $K$. Since $K$ is a fixed point of $\mathcal{B}$, Proposition
~\ref{p.l.minimality} implies that
$\overline{A_{\mathrm{t}}}\subset K$. 
Since by definition of a strict attractor the set  $K$ attracts every compact set 
in a neighbourhood of it, the minimum fixed point $\overline{A_{\mathrm{t}}}$
is attracted by $K$. Therefore $\overline{A_{\mathrm{t}}}=K$, proving the proposition. 
\end{proof}

The following example  shows that there are 
 IFSs with $S_{\mathrm{t}}\neq \emptyset$
without 
strict attractors. In this example $\overline{A_{\mathrm{t}}}$ is stable (recall \eqref{e.stableBH}).

\begin{example}
\label{example1}{\emph{
Consider the maps $T_{1},\, T_2\colon[0,2]\to [0,2]$ 
depicted in Figure \ref{f.nconley} and
defined by
\begin{itemize}
\item
$T_{1}(x)=\frac{1}{3}x$ and 
\item
 $T_{2}\colon[0,2]\to [0,2]$ is the piecewise-linear map defined by
$T_{2}(x)=\frac{1}{3}x+\frac{2}{3}$ for $x\in[0,1]$ and $T_{2}(x)=x$ for 
$x\in [1,2]$,
\end{itemize}}}

{\emph{
  Let $\mathcal{C}$ be the standard ternary Cantor set in the interval $[0,1]$.
  We claim that 
  $S_{\mathrm{t}}\ne \emptyset$, $\overline{A_{\mathrm{t}}}=\mathcal{C}$,
and  $\mathcal{C}$ is not a strict attractor. Indeed, $\mathcal{C}$ is not a Conley attractor.
We now prove these assertions. 
}} 

{\emph{
First, as
 $T_{1}$ is a contraction $\bar 1\in S_{t}$,
 where $\bar 1$ is the sequence whose terms are all equal to $1$.}}

{\emph{
For the second assertion, consider the auxiliary
$\mathrm{IFS}(f_{1},f_{2})$
where $f_{1}=T_{1|[0,1]}$ and $f_{2}=T_{2|[0,1]}$.
Note 
that $\mathcal{C}$ is the attractor of  $\mathrm{IFS}(f_{1},f_{2})$ (see, for instance,
Example 1 in  \cite[Section 3.3]{Hu}). 
In particular, the set 
$\mathcal{C}$ is the unique fixed point of the Barnsley-Hutchinson operator 
$\mathcal{B}$ of  $\mathrm{IFS}(T_{1},T_{2})$ contained in $[0,1]$. Since $[0,1]$
is $\mathcal{B}$-invariant, by Propositions~\ref{p.existence} and \ref{p.l.minimality}  we have 
$\overline{A_{\mathrm{t}}}\subset [0,1]^* \subset[0,1]$
and therefore $\overline{A_{\mathrm{t}}}=\mathcal{C}$.}}

{\emph{To see that $\mathcal{C}$ is not a strict
attractor, just note that every open neighbourhood of $\mathcal{C}$ necessarily contains an interval 
of the form $[1,\delta)$. Since $T_{2}(x)=x$ for all $x\in [1,\delta)$ the assertion
follows.}}

  \begin{figure}[!h]
  \centering
  \begin{tikzpicture}[xscale=1.3,yscale=1.3]
      \draw[->] (1,0) -- (2,0);
    \draw[->] (0,1) -- (0,2);
    \draw[green,smooth,samples=100,domain=1.0:2.0] plot(\x,{\x});
     \draw[-] (0,0) -- (1,0);
      \draw[-] (0,0) -- (0,1);
       \draw[dashed] (2,0) -- (2,2);
       \draw[dashed] (0,2) -- (2,2);
        \draw[dashed] (0,1) -- (1,1); 
         \draw[dashed] (1,0) -- (1,1);
       \draw[green,smooth,samples=100,domain=0.0:1.0] plot(\x,{0.33333*\x+0.66666});
     \draw[-] (0,0) -- (2,0);
      \draw[-] (0,0) -- (0,2);
      \draw[blue,smooth,samples=100,domain=0.0:2.0] plot(\x,{0.33333*\x});
       \node[scale=1,left] at (0,0.66666) {$\frac{2}{3}$};
       \end{tikzpicture}
     \caption{The set $\overline{A_{\mathrm{t}}}$ is not a Conley attractor}
    \label{f.nconley}
     \end{figure}
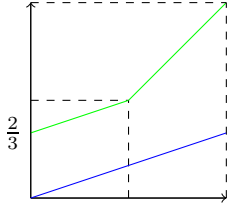
\end{example}

\subsection{Proof of Theorem~\ref{mt.local}}
\label{ss.mtlocal}
Since every strict attractor is a Conley attractor, to prove the theorem
it is enough to see that given  an $\mathrm{IFS}(T_{1},\dots T_{k})$
 such that $\overline{A_{\mathrm{t}}}$ is a non-empty
 Conley attractor then $\overline{A_{\mathrm{t}}}$ is a strict attractor.
  We need the following preparatory lemma: 
 \begin{lema}
\label{l.era14}
 Consider sequences 
 $(A_n)$ of compact sets in $\mathcal{H}(X)$ and
 $(p_n)$ of points in $X$
 with
 $A_n\to A$ and $p_n\to p$ in the  
 Hausdorff distance $d_{H}$. 
 Then 
 $$
 d (p,A)=\displaystyle\lim_{n\rightarrow \infty} d(p_{n},A_{n}).
 $$
 \end{lema}
 
 \begin{proof}
 We use  the following ``triangular'' inequality:
 given a point $q$ and two compact sets $A$ and $B$ it holds
 $$
 d(q,A)\le d(q,B)+ d_H(A,B).
$$ 
Consider the sequences $(A_n)$ and $(p_n)$ in the lemma. 
 Applying twice the ``triangular'' inequality
 above we get
$$
 d (p,A)\leq d(p,p_{n})+d(p_{n},A)\leq  d(p,p_{n})+d(p_{n},A_{n})+d_H(A_{n},A).
$$
By hypothesis,  $d(p,p_{n})\to 0$ and $d_{H}(A_{n},A)\to 0$.  
We conclude that 
\begin{equation}
 \label{eq.liminf}
d(p,A)\leq \liminf_{n} d(p_{n},A_{n}).
\end{equation}

Applying again twice the ``triangular'' inequality, we get
 $$ 
 d (p_{n},A_{n})\leq 
 d(p_{n},p)+d(p,A_{n})\leq d(p_{n},p)+d(p,A)+ d_H(A,A_{n}).
 $$
This implies that
\begin{equation}
 \label{eq.limsup}
 \limsup_{n} d(p_{n},A_{n})\leq d(p,A).
 \end{equation}
 Equations \eqref{eq.liminf} and \eqref{eq.limsup} imply the lemma.
 \end{proof}

 We are now ready to prove the theorem. Since $\overline{A_{\mathrm{t}}}$ is a Conley attractor it
 has an open neighbourhood  $U$  such that 
$ \mathcal{B}^{n}(\overline{U})\to\overline{A_{\mathrm{t}}}$.
To prove that
$\overline{A_{\mathrm{t}}}$ is a strict attractor  we need  to check that
for every compact set 
$K\in \mathcal{H}(\overline{U})$ it holds
$\mathcal{B}^{n}(K)\rightarrow \overline{A_{\mathrm{t}}}$. 
For that
it is enough to see that for any $\epsilon >0$ there is $n_{0}\in \mathbb{N}$ such that 
for every $n\geq n_{0}$ it holds
 \begin{equation}
 \label{e.toproveremains}
 d_{H}(\overline{A_{\mathrm{t}}},\mathcal{B}^{n}(K))=\max\{ h_{s}(\overline{A_{\mathrm{t}}},\mathcal{B}^{n}(K))
 ,  h_{s}(\mathcal{B}^{n}(K),\overline{A_{\mathrm{t}}})\}
 \leq\epsilon.
 \end{equation}

By hypothesis, $\mathcal{B}^{n}(\overline{U})\to 
\overline{A_{\mathrm{t}}}$.  Thus there is $n_{0}$
such that for every  $n\geq n_{0}$
we have
$$
 h_{s}(\mathcal{B}^{n}(\overline{U}),\overline{A_{\mathrm{t}}})
\leq\epsilon.
$$
Therefore, for every $n\ge n_0$,
$$
 h_{s}(\mathcal{B}^{n}(K),\overline{A_{\mathrm{t}}})\leq 
 h_{s}(\mathcal{B}^{n}(\overline{U}),\overline{A_{\mathrm{t}}})
\leq\epsilon.
$$
Hence  to prove \eqref{e.toproveremains} it remains to see that $
h_{s}(\overline{A_{\mathrm{t}}},\mathcal{B}^{n}(K))
\le \epsilon $ for
every $n$ sufficiently large. This is proved in the next lemma.

 \begin{lema}\label{l.atracaofatal}
For every $K\in \mathcal{H}(X)$ it holds $\lim_{n\to \infty} h_{s}(\overline{A_{\mathrm{t}}}
,\mathcal{B}^{n}(K))=0$.
\end{lema}

\begin{proof}
The proof is by contradiction. Assume 
 that there are a compact set $K\in \mathcal{H}(X)$ and a sequence $(n_\ell)$ such that
$h_{s}(\overline{A_{\mathrm{t}}},\mathcal{B}^{n_{\ell}}(K))> \epsilon$ for every
$\ell$. 
Note that for each $\ell$ 
there is $p_{n_\ell}\in \overline{A_{\mathrm{t}}}$ with $d(p_{n_\ell},\mathcal{B}^{n_\ell}(K))> \epsilon$. 
By compactness we can assume that
$p_{n_{\ell}}\rightarrow p^{*}\in \overline{A_{\mathrm{t}}}$ 
and that $\mathcal{B}^{n_{\ell}}(K)\rightarrow  \widehat{K}$.
By Lemma~\ref{l.era14}, 
\begin{equation}
\label{e.acontradicao}
d(p^{*},\widehat{K})\geq \epsilon.
\end{equation}

We now derive a contradiction from this inequality.
By construction, there is 
 $\ell_{0}$ such that 
\begin{equation}
\label{e.byconstruction}
h_{s}(B^{n_{\ell}}(K),\widehat{K})<\frac{\epsilon}{2},
\quad
\mbox{for all $\ell\geq \ell_0$.}
\end{equation}
Take $q\in B_{\frac{\epsilon}{2}}(p^{*})\cap A_{\mathrm{t}}$ and note that
there is a sequence $\omega=\omega_{0}\
\omega_{1}\ldots\in S_\mathrm{t}$ such that 
$$
\bigcap_{n\geq 0} T_{\omega_{0}} \circ\cdots\circ T_{\omega_{n}}(X)=\{q\}.
$$
Therefore there is $m_{0}$  such that
$$
T_{\omega_{0}} \circ\cdots\circ T_{\omega_{m-1}}(K)\subset
T_{\omega_{0}} \circ\cdots\circ T_{\omega_{m-1}}(X)\subset B_{\frac{\epsilon}{2}}(p^{*})
\quad
\mbox{for every  $m\geq m_{0}$}.
$$
Since $T_{\omega_{0}} \circ\cdots\circ T_{\omega_{m-1}}(K)\subset  \mathcal{B}^{m}(K)$,
for every  $\ell$ big enough we have 
$\mathcal{B}^{n_{\ell}}(K)\cap B_{\frac{\epsilon}{2}}(p^{*})\neq \emptyset$. 

Note that for every $\ell$ sufficiently large 
$\mathcal{B}^{n_{\ell}}(K)\cap B_{\frac{\epsilon}{2}}(p^{*})\ne \emptyset$ and equation 
\eqref{e.byconstruction} holds.
Hence for every
$z\in \mathcal{B}^{n_{\ell}}(K)\cap B_{\frac{\epsilon}{2}}(p^{*})$
we have $d (z,\widehat{K})<\frac{\epsilon}{2}$ and $d(z,p^{*})<\frac{\epsilon}{2}$.
Hence $d(p^{*},\widehat{K})<\epsilon$ contradicting \eqref{e.acontradicao}. 
This ends the proof of the lemma.
\end{proof}

The proof of the theorem is now complete.
\hfil \qed

\begin{schol}\label{sc.atracaofatal2}
If $U$ is a neighbourhood of  $\overline{A_{\mathrm{t}}}$ such that
$\mathcal{B}^n(\overline{U}) \to  \overline{A_{\mathrm{t}}}$ then every compact subset
of $\overline{U}$ also satisfies $\mathcal{B}^n(K) \to  \overline{A_{\mathrm{t}}}$.
\end{schol}
 
 We have the following corollary that allows us to stablish a   
 connection between  the set $\overline{A_{\mathrm{t}}}$ and  semifractals.
 
\begin{coro}\label{c.l.atractinside}
Consider an IFS such that $S_{t}\neq \emptyset$. Then 
$$
\lim_{n\to \infty}\mathcal{B}^{n}(K)= \overline{A_{\mathrm{t}}}, \quad 
\mbox{for every 
compact set $K\subset \overline{A_{\mathrm{t}}}$.}
$$
\end{coro}

\begin{proof}
The statement is an immediate consequence of Lemma~\ref{l.atracaofatal} and the invariance of 
$\overline{A_{\mathrm{t}}}$.
\end{proof}
 
 \begin{remark}\label{r.putting}{\em{
 Combining  Propositions~\ref{p.l.minimality} and
 Corollary
\ref{c.l.atractinside} one gets the following:
  if $S_{\mathrm{t}}\ne \emptyset$ then
   set $\overline{A_{\mathrm{t}}}$ is a minimum fixed point that
attracts every compact set inside it.}}
\end{remark}

   


 

 
 



 \subsection{Proof of Theorem~\ref{mt.att}} 
 \label{ss.attractor}
Suppose that the set $A_{\mathrm{t}}$ is non-empty. We need to prove the equivalence of
the  following three assertions: 
\begin{enumerate}
 \item 
$\overline{A_\mathrm{t}}=X^*$;
 \item
the Barnsley-Hutchinson operator $\mathcal{B}$ has a unique fixed point;
\item
$X^{*}$ is a global attractor
(a strict attractor whose basin is the whole space).
\end{enumerate}

The equivalence $1\Leftrightarrow 2$ follows immediately
from Proposition~\ref{p.l.minimality} (``the minimum fixed point $\overline{A_\mathrm{t}}$ is equal to
the maximum fixed point $X^*$'').

The implication $3 \Rightarrow 2$ follows noting that if $K$ is a fixed point
of $\mathcal B$ and since $X^*$ is a global attractor then  $K=\mathcal{B}^n(K)\to X^*$ and thus $K=X^*$.

To prove 
 $1\Rightarrow 3$ note that, by Lemma~\ref{l.semnome}, $X^{*}=\lim_{n\to\infty}\mathcal{B}^{n}(X)$ and thus $X^*$ is a Conley attractor. Then if $\overline{A_{t}}=X^{*}$  we 
 have that $\overline{A_{\mathrm{t}}}$ is a Conley attractor, by Theorem \ref{mt.local}
 and Scholium ~\ref{sc.atracaofatal2} this set is
 a strict attractor whose basin is the whole space.
 \hfil \qed

  \subsection{Structure of set $A_{\mathrm{t}}$} 
  \label{ss.structure}
The main result of this section is Proposition~\ref{p.dichotomy}
about the topological structure of the target set $A_{\mathrm{t}}$. This result will be used in 
Section~\ref{ss.separability}.

We begin by observing that, in general, the set $A_{\mathrm{t}}$ is not necessarily closed.
The IFS in Example \ref{example1}  illustrates this case.
In this example $\overline{A_{\mathrm{t}}}$ is the 
ternary Cantor set $\mathcal{C}$ in $[0,1]$, thus $1\in  \overline{A_{\mathrm{t}}}$.
We claim that
$1\not\in \mathrm{A_{t}}$ and thus $\mathrm{A_{t}}$ is not closed.  
Recall the definitions of $\mathrm{IFS}(T_{1},T_{2})$ and $\mathrm{IFS}(f_{1},f_{2})$ in this example 
and consider 
their  natural associated  projections 
 $\pi_{T}$ and $\pi_f$, see \eqref{e.pi}.
Arguing by contradiction, if $1\in A_{\mathrm{t}}$ then there is a sequence $\xi\in S_{\mathrm{t}}$ with 
 $\pi_{T}(\xi)=1$. In this case we also have $\pi_{f}(\xi)=1$. 
It is easy to check that  $\xi=\bar 2$ and that
$\bar 2\not\in S_{\mathrm{t}}$, where $\bar 2=(\xi_i=2)$. This gives a contradiction.

\begin{prop}\label{p.dichotomy}
Consider  
 $\mathrm{IFS}(T_1,\dots, T_k)$  defined on a compact set $X$
such that $A_{\mathrm{t}}\neq \emptyset$. 

\begin{enumerate}
 \item Assume that the IFS is injective in  $A_{\mathrm{t}}$. Then
 either 
 $A_{\mathrm{t}}$ is a singleton
 or $A_{\mathrm{t}}$ has no isolated points (thus it is infinite).  
 \item
Assume that the maps
$T_{i}$ are open.
 Then either $A_{\mathrm{t}}$ 
has empty interior
or $\mathrm{int}(A_{\mathrm{t}})\subset A_{\mathrm{t}}\subset\overline{\mathrm{int}(A_{\mathrm{t}})}$.
\end{enumerate}
\end{prop}

We observe that in the proof of the first item of the proposition we only use the injectivity
of the maps $T_{i}$ on $A_{\mathrm{t}}$.

Let us also observe that if the maps $T_{i}$ are not injective then  
the set $A_{\mathrm{t}}$ can be finite with more than one element. The maps
depicted  in Figure \ref{f.figure11} 
give an example of this case, 
where 
$A_{\mathrm{t}}=\{0,\frac{1}{2},1\}$.
  
  \begin{figure}[!h]
  \centering
  \begin{tikzpicture}[xscale=2,yscale=2]
     \draw[->] (0,0) -- (1,0);
      \draw[->] (0,0) -- (0,1);
       \draw[green,smooth,samples=100,domain=0.0:0.5] plot(\x,{\x+0.5});
       \draw[green,smooth,samples=100,domain=0.5:1.0] plot(\x,{1});
      \draw[blue,smooth,samples=100,domain=0.5:1.0] plot(\x,{\x-0.5});
       \draw[blue,smooth,samples=100,domain=0.0:0.5] plot(\x,{0});
       \node[scale=1, right] at (1,0.5) {$\frac{1}{2}$};
       \node[scale=1, below] at (0.5,0) {$\frac{1}{2}$};
       \end{tikzpicture}
     \caption{$\#(A_{\mathrm{t}})=3$} \label{f.figure11}
  \end{figure}


\begin{remark}{\em{
Every injective $\mathrm{IFS}(T_{1},\dots ,T_{k})$ defined on $[0,1]$ 
satisfies
the hypotheses in the second part of
Proposition~\ref{p.dichotomy}. 
}}
\end{remark}

\begin{proof}[Proof of Proposition~\ref{p.dichotomy}]
We prove the first item in the proposition. If $A_{\mathrm{t}}$ is a singleton we are done. 
Otherwise $\# (A_{\mathrm{t}})\geq 2$.
To  see that every $p\in A_{\mathrm{t}}$ is not isolated we check 
that for every 
every neighbourhood $V$ of $p$ the set $A_{\mathrm{t}}\cap V$
contains at least two points.
By definition of $A_{\mathrm{t}}$, there is a finite 
sequence $\xi_{0}\ldots \xi_{n}$ such that 
$$
T_{\xi_{0}}\circ\cdots\circ T_{\xi_{n}}(X)\subset V.
$$
In particular, 
$$
T_{\xi_{0}}\circ\cdots\circ T_{\xi_{n}}(A_{\mathrm{t}})\subset V.
$$
Since $T_{\xi_{0}}\circ\cdots\circ T_{\xi_{n}}(A_{\mathrm{t}})\subset A_{\mathrm{t}}$ 
(recall Scholium~\ref{sc.Atinv})
and  $T_{\xi_{0}}\circ\cdots\circ T_{\xi_{n}}$ is one-to-one in $A_{t}$, 
we  have that $V$ contains at least two points of $A_{t}$,
proving the first part of the proposition.

We now prove the second item of the proposition.
If $A_{\mathrm{t}}$ has empty interior we are done. 
Thus we can  assume
that $\mathrm{int}(A_{\mathrm{t}})\ne \emptyset$. Since 
$\mathrm{int}({A_{\mathrm{t}}})\subset {A_{\mathrm{t}}}$
it only remains to see 
that $A_{\mathrm{t}}\subset\overline{\mathrm{int}(A_{\mathrm{t}})}$.
Take a point $x\in A_{\mathrm{t}}$ and any open neighbourhood $V$ of $x$. 
By definition of  $A_{\mathrm{t}}$ there is a finite sequence $\xi_{0}\dots \xi_{n}$
such that 
$$
T_{\xi_{0}}\circ\cdots\circ T_{\xi_{n}}(X)\subset V.
$$
By $T_i ({A_{\mathrm{t}}})\subset {A_{\mathrm{t}}}$ it follows
$$
T_{\xi_{0}}\circ\cdots\circ T_{\xi_{n}}(\mathrm{int}(A_{\mathrm{t}}))\subset V\cap A_{\mathrm{t}}.
$$
Since $\mathrm{int}(A_{\mathrm{t}})$ is an open set and the maps $T_i$ are open, then
 $T_{\xi_{0}}\circ\cdots\circ T_{\xi_{n}}(\mathrm{int}(A_{\mathrm{t}}))$ is a
 non-empty and open
 subset of $V \cap A_{\mathrm{t}}$, thus 
 $V\cap \mathrm{int}(A_{\mathrm{t}})\neq \emptyset$.
 Since this holds for every neighbourhood of $x$ we get that 
$x\in \overline{\mathrm{int}(A_{\mathrm{t}})}$.
The proof of the proposition is now complete.
\end{proof}

\subsection{Proof of Theorem~\ref{mt.jogodocaos}} 
\label{ss.mtjogo}
%
%
%
%
%
Suppose that $\overline{A_{\mathrm{t}}}$ is stable.
We need to prove that given any disjunctive sequence $\xi$ and any point $x$
it holds 
$$
\overline{A_{\mathrm{t}}}=\bigcap_{\ell \geq 0} \overline{\{x_{n,\xi}\colon n\geq \ell\}},
\quad \mbox{where}\quad  x_{n,\xi}\eqdef T_{\xi_{n}}\circ\dots \circ T_{\xi_{0}}(x).
$$
To simplify notation write 
$$
Y_\ell\eqdef \{x_{n,\xi}\colon n\geq \ell\}.
$$

For the inclusion ``$\subset$'' take any point $p\in A_{\mathrm{t}}$ and fix $\ell\geq 0$. We need to see 
that for every  neighbourhood $V$ of $p$ it holds
\begin{equation}\label{e.inclusion0205}
V \cap Y_\ell\ne \emptyset.
\end{equation}
By definition of $A_{\mathrm{t}}$ there is a finite sequence 
$c_{0}\dots c_{r}$ such that 
\begin{equation}\label{e.abril16}
T_{c_{r}}\circ \dots \circ T_{c_{0}}(X)\subset V.
\end{equation}
 We can  assume that $r\geq \ell$. Since
  $\xi$ has dense orbit there is $m_{1}$ such that 
  $$
  \xi_{m_{1}}=c_{0},\, \xi_{m_{1}+1}=c_{1},\,
  \dots,\,\xi_{m_{1}+r}=c_{r}.
 $$ 
   Therefore, from \eqref{e.abril16} it follows 
   $$
  x_{m_{1}+r,\xi}=T_{\xi_{m_{1}+r}}\circ\dots\circ T_{\xi_{m_{1}}}
  \circ T_{\xi_{m_{1}-1}}\circ \dots\circ T_{\xi_{0}}(x)\in V.
  $$
 Since $m_{1}+r\geq \ell$ we have that $V\cap Y_\ell 
 \neq \emptyset$, proving 
 \eqref{e.inclusion0205}. 
 
 We now prove the inclusion ``$\subset$''.
  Take any neighbourhood $V$ of $\overline{A_{\mathrm{t}}}$. Since 
$\overline{A_{\mathrm{t}}}$ is stable it has a neighbourhood $V_{0}\subset V$
such that $\mathcal{B}^{n}(V_{0})\subset V$ for every $n\geq 0$.
Since $\xi$ is a disjunctive sequence and $A_{\mathrm{t}}\subset V_{0}$ 
there is $n_{0}\in \mathbb{N}$ such that $x_{n_{0},\xi}\in V_{0}$.
Hence 
$Y_{n_0} \subset V$ and thus
 $\overline{Y_{n_0}}\subset \overline{V}$.
 As the sequence of sets 
 $(\overline{Y_\ell})$
 is nested,  we have that 
 $
 \bigcap_{\ell \geq 0} \overline{Y_\ell}\subset \overline{V}.
 $
 Since this holds for every neighbourhood $V$ of $\overline{A_{\mathrm{t}}}$ we conclude that 
  $$
 \bigcap_{\ell \geq 0} \overline{Y_\ell}
 \subset \overline{A_{\mathrm{t}}} .
 $$
  Finally,
 as $\overline{Y_\ell}$ is a nested sequence 
 of
  compact sets, from Lemma~\ref{l.semnome} and the definition of a Hausdorff limit, 
  it follows
 $Y_\ell {\longrightarrow}  \overline{A_{\mathrm{t}}}$,
 where the convergence is in 
the Hausdorff distance.
 \hfil 
\qed

\section{Measure and rigidity of $S_{\mathrm{t}}$ for IFSs on $[0,1]$}
\label{s.asstableifs01}
In this section,
for
IFSs 
defined on $[0,1]$,
 we study the measure of   $S_{\mathrm{t}}$ for Markov measures
and prove 
Theorem~\ref{mt.abundance}, see Section~\ref{ss.mtabundance}. In Section~\ref{ss.prigidtyst}
we prove a result about probabilistic rigidity of the set $S_{\mathrm{t}}$:
under quite general conditions,   if  $S_{\mathrm{t}}$ intersects the support of a Markov measure
it has full probability. Finally, in Section~\ref{ss.separability} we characterise separable IFSs. 
 
\subsection{Proof of Theorem~\ref{mt.abundance}}
\label{ss.mtabundance}
Given an $\mathrm{IFS}(T_{1},\dots T_{k})$ defined on $I=[0,1]$ we need to see that
every
mixing Markov measure that splits the IFS in some non-trivial closed interval $J$
satisfies $\mathbb{P}^{+}(S_{\mathrm{t}})=1$.

Recall the definition
 of the fibre $I_\xi$ of a sequence in $\Si_k^+$ in \eqref{spine}.
Given $x\in [0,1]$ we consider  the set of sequences whose fibres contain $x$ defined by 
\begin{equation}\label{e.sigmax}
\Sigma_x\eqdef
\{\xi\in \Sigma_k^+ \colon x\in I_\xi\}.
\end{equation}
 
  \begin{lema}\label{l.lemasuficiente}
 Suppose that $\mathbb{P}^+(\Sigma_x)=0$ for all $x\in [0,1]$.  Then
 $\mathbb{P}^+(S_{\mathrm{t}})=1$.
 \end{lema}
  
 \begin{proof}
 Note that if
$\xi\not\in S_\mathrm{t}$ then its fibre $I_\xi$ is a 
non-trivial interval and hence contains 
a rational point. This implies that
$$
( S_{\mathrm{t}})^c=
\Sigma_{k}^{+}\setminus S_{\mathrm{t}}\subset \bigcup_{x\in \mathbb{Q}\cap [0,1]}\Sigma_{x}.
$$
This union is countable and each set $\Sigma_x$ satisfies $\mathbb{P}^{+}(\Si_x)=0$, thus
$\mathbb{P}^{+}(S_{\mathrm{t}})=1$.
 \end{proof}
 
 In view of Lemma~\ref{l.lemasuficiente}, to see that $\mathbb{P}^+(S_{\mathrm{t}})=1$ it is sufficient to
 show the following:
 
 \begin{teo}
 \label{t.p.usadoAtdois}
Consider an
$\mathrm{IFS}$ defined on $I=[0,1]$ and a mixing Markov measure
$\mathbb{P}^{+}$ that splits the $\mathrm{IFS}$ in some non-trivial interval $J$. 
Then  $\mathbb{P}^+(\Sigma_x)=0$ for all $x\in [0,1]$.
\end{teo}
 
 \begin{proof}
   By the
 splitting hypothesis there is a pair of admissible cylinders $[i_{1}\dots i_{\ell}]$ and 
 $[j_{1}\dots j_{s}]$ with $i_{1}=j_{1}$ such that 
 \begin{equation}
 \label{e.splittingcondition}
 \begin{split}
 &T_{j_{1}}\circ\dots\circ T_{j_{s}}(I)\cap T_{i_{1}}\circ\dots \circ T_{i_{\ell}}(I)=\emptyset,
 \quad\mbox{and}\\
  &T_{j_{1}}\circ\dots\circ T_{j_{s}}(I) \cup T_{i_{1}}\circ\dots \circ T_{i_{\ell}}(I)
  \subset J.
\end{split}
\end{equation}
%
Next claim restates the splitting condition: 
 \begin{claim} \label{c.claimera52}
 There are
 admissible cylinders $[\xi_{0}\dots \xi_{N-1}]$
 and $[\omega_{0}\dots\omega_{N-1}]$
  such that $\xi_{0}=\omega_{0}$, $\xi_{N-1}=\omega_{N-1}$, 
\[
 \begin{split}
& T_{\xi_{0}}\circ\cdots\circ T_{\xi_{N-1}}(I)\cap 
T_{\omega_{0}}\circ\cdots\circ T_{\omega_{N-1}}(I)=\emptyset
\quad \mbox{and}\\
& T_{\xi_{0}}\circ\cdots\circ T_{\xi_{N-1}}(I)\cup 
T_{\omega_{0}}\circ\cdots\circ T_{\omega_{N-1}}(I)\subset J.
\end{split}
\]
\end{claim}

\begin{proof}
Consider $j_1,\dots,j_s$ and $i_1,\dots,i_\ell$ as in \eqref{e.splittingcondition}.
 Since $\mathbb{P}^{+}$ is mixing 
 there is $n_{0}$ such that
 for every $n\geq n_{0}$ there are 
 admissible cylinders of the form
  $[i_{\ell}c_{1}\dots c_{n-1}0]$ and $[j_{s}d_{1}\dots d_{n-1}0]$. 
 Take now $n_{1},n_{2}\geq n_{0}$ 
  and admissible cylinders
  $[i_{\ell}c_{1}\dots c_{n_{1}}0]$ and $[j_{s}d_{1}\dots d_{n_{2}}0]$ such that $n_{1}+\ell=n_{2}+s$. 
  Let $N=\ell+n_{1}+1$. Then the cylinders 
   $$
  [\xi_{0}\dots \xi_{N-1}]=[i_{1}\dots i_{\ell}c_{1}\dots c_{n_{1}}0] \quad\mbox{and} \quad
[\omega_{0}\dots \omega_{N-1}]=[j_{1}\dots i_{s}d_{1}\dots d_{n_{2}}0]
$$  
are admissible and satisfy the intersection and union properties in the claim. To see why this is so
 note that
$T_{c_1} \circ \cdots \circ  T_{c_{n_1}}\circ T_0(I)\subset I$ and 
$T_{d_1} \circ \cdots \circ  T_{d_{n_2}}\circ T_0(I)\subset I$. 
\end{proof}

We now
fix $x\in I$ and prove that $\mathbb{P}^+(\Sigma_x)=0$.
For that
fix $N$, 
the admissible cylinders $[\xi_{0}\dots \xi_{N-1}]$
 and $[\omega_{0}\dots\omega_{N-1}]$ in the claim, and
for $j\geq 1$ define the sets  
 \begin{equation}\label{e.sigmaxj}
 \Sigma_{x}^{j}\eqdef\{[a_{0}\ldots a_{jN-1}]\subset
  \Sigma_{k}^+\colon x\in T_{a_{0}}\circ\cdots\circ
T_{ a_{jN-1}}(I)\}
\quad
\mbox{and}
\quad
 S_{x}^{j}\eqdef
  \bigcup_{C
 \in\Sigma_{x}^{j}}C.
\end{equation}
Note that by definition $S_x^{j+1}\subset S_x^j$ and that
for  each  $j\geq 1$ it holds
$\Sigma_x\subset S_{x}^{j}$. Hence
$$
\Sigma_x \subset \bigcap_{j\ge 1} S^j_x.
$$
Therefore
$$
\mathbb{P}^+ (\Sigma_x) \le 
\mathbb{P}^+ \Big( \bigcap_{j\ge 1} S^j_x \Big) =\lim_{j\to \infty} \mathbb{P}^+ (S^j_x).
$$
Hence the assertion $\mathbb{P}^+ (\Sigma_x)=0$ in the theorem  follows from the next proposition:
   
 \begin{prop}\label{p.zabelerolimit}
  $\lim_{j\to \infty} \mathbb{P}^{+}(S_{x}^{j})=0$. 
 \end{prop}

\begin{proof}
Suppose, for instance, that the cylinders in the claim satisfy
\begin{equation}\label{e.forinstance}
0<\mathbb{P}^{+}([\xi_{0}\ldots\xi_{N-1}])
\leq
 \mathbb{P}^{+}([\omega_{0}\ldots\omega_{N-1}]).
 \end{equation}  
 The first inequality follows form the admissibility of
 $[\xi_{0}\dots \xi_{N-1}]$. 
 
 Define
  for $j\geq 1$ the family of cylinders
 $$
 E^{j}\eqdef\{[a_{0}\ldots a_{jN-1}]\subset
  \Sigma_{k}^+\colon  \sigma^{iN}([a_{0}\ldots a_{jN-1}])
  \cap [\xi_{0}\ldots\xi_{N-1}]=\emptyset, \,\,i=0,\ldots ,j-1\} 
 $$
 and their union
$$
Q^{j}\eqdef\bigcup_{C
 \in E^{j}} C. 
$$
Note that by definition $Q^{j+1} \subset Q^j$.
Let 
$$
Q^{\infty}\eqdef \bigcap_{j\geq 1} Q^{j}=
\{\omega\in \Sigma^+ \colon \sigma^{iN}(\omega)
 \cap[\xi_{0}\ldots\xi_{N-1}]=\emptyset\,\, \mbox{for all}\,\, i\geq 0\}. 
 $$
Recall that the mixing property of  $(\sigma,\mathbb{P}^{+})$ implies the ergodicity of $(\sigma^{N},\mathbb{P}^{+})$. Thus
the Birkhoff's ergodic theorem implies that $\mathbb{P}^{+}(Q^{\infty})=0$.
 Therefore condition $Q^{j+1}\subset Q^{j}$ implies that 
 $$
 \lim_{j\to \infty} \mathbb{P}^{+}( Q^{j})=0.
 $$
In view of this property, the proposition follows from the next lemma.   
 \begin{lema} \label{l.menorouigual}
  $\mathbb{P}^{+}(S_{x}^{j})\leq \mathbb{P}^{+}(Q^{j})$ 
 for all $j\geq 1$.
 \end{lema}

\begin{proof}
 For each $j\geq 1$ consider the auxiliary substitution function 
 $F_{j}: \Sigma_{x}^{j}\to E^{j}$ defined as follows.
 For each cylinder $[\alpha_0\dots \alpha_{jN-1}]\in \Sigma_{x}^{j}$ 
 we consider its sub-cylinders
 $[\alpha_0\dots \alpha_{N-1}],[\alpha_N\dots \alpha_{2N-1}], \dots, 
 [\alpha_{(j-1)N}\dots \alpha_{jN-1}]$
 and use the following  concatenation notation
 $$
  [\alpha_0\dots \alpha_{jN-1}]=
   [\alpha_0\dots \alpha_{N-1}]\ast [\alpha_N\dots \alpha_{2N-1}]\ast
 \cdots  \ast [\alpha_{(j-1)N}\dots \alpha_{jN-1}].
 $$
 In a compact way, we write
 $$
 C=C_0 \ast C_1 \ast \cdots \ast C_{j-1}
 $$
 where the cylinder  $C$ has size $jN$ and each cylinder $C_i$ has size
 $N$.
 With this notation we define $F_j$ by
 $$
 F_j(C)\eqdef
 F_j(C_0 \ast C_1 \ast \cdots \ast C_{j-1})=
 C_0' \ast C_1' \ast \cdots \ast C_{j-1}',
 $$
 where $C_i'=C_i$ if $C_i\ne [\xi_0\dots \xi_{N-1}]$ and 
 $C_i'=[\omega_0\dots \omega_{N-1}]$ otherwise.

\begin{claim}\label{c.substitutionmap1}
For every $j\ge 1$ it holds 
$
\mathbb{P}^{+}(C)\leq \mathbb{P}^{+} (F_{j}(C)) \quad \textrm{for every}
\quad C\in \Sigma_{x}^{j}.
$ 
\end{claim}  

\begin{proof}
Recalling that
 $\omega_{0}=\xi_{0}$ and $\omega_{N-1}=\xi_{N-1}$,
 from equation~\eqref{e.forinstance}
 we immediately get the following:
 For every  $m, s\ge 0$ and  every pair of cylinders $[a_{0}\ldots a_{s}]$ and 
 $[b_{0}\ldots b_{m}]$ it holds
\begin{enumerate} 
\item 
 $\mathbb{P}^{+}([a_{0}\ldots a_{s}\xi_{0}\ldots\xi_{N-1}
 b_{0}\ldots b_{m}])\leq 
 \mathbb{P}^{+}([a_{0}\ldots a_{s}\omega_{0}\ldots\omega_{N-1}
 b_{0}\ldots b_{m}])$,
 \item$
  \mathbb{P}^{+}([\xi_{0}\ldots\xi_{N-1}
 b_{0}\ldots b_{m}])\leq 
 \mathbb{P}^{+}([\omega_{0}\ldots\omega_{N-1}
 b_{0}\ldots b_{m}])$, and
 \item
  $
  \mathbb{P}^{+}([a_{0}\ldots a_{s}\xi_{0}\ldots\xi_{N-1}]\leq 
 \mathbb{P}^{+}([a_{0}\ldots a_{s}\omega_{0}\ldots\omega_{N-1}].
  $
  \end{enumerate}
The inequality $\mathbb{P}^{+}(C)\leq \mathbb{P}^{+} (F_{j}(C))$ now follows from the definition of
$F_j$.
 \end{proof}

%
  
\begin{claim}\label{c.substitutionmap}
The map $F_{j}$ is injective for every $j\ge 1$.
\end{claim}  

\begin{proof}
Fix $j\geq 1$. 
  Given cylinders $C, \widetilde C\in \Sigma_x^j$,
  using the notation above write
  $ C=C_0 \ast C_1 \ast \cdots \ast C_{j-1}$ 
  and $\widetilde C =\widetilde C_0 \ast \widetilde C_1 \ast \cdots  \ast \widetilde C_{j-1}$.
 Then
 $$
 F_j(C)=C_0' \ast C_1' \ast \cdots \ast C_{j-1}'
 \quad
 \mbox{and}
 \quad
 F_j(\widetilde C )=\widetilde C_0' \ast \widetilde C_1' \ast \cdots  \ast \widetilde C_{j-1}'.
$$
 Suppose that
$F_{j}(C)=F_{j}(\widetilde{C})$.
Then $C_i'=\widetilde C_i'$ for all $i=0,\dots, N-1$. If $C\ne \widetilde C$
there is a first  $i$ such that $C_i\ne \widetilde C_i$. 
Then either $C_i= [\xi_0\dots \xi_{N-1}]$ and 
$\widetilde C_i= [\omega_0\dots \omega_{N-1}]$  or vice-versa. Let us assume that the
first case occurs.

If  $i=0$ then 
the definition of $\Sigma_{x}^{j}$ 
implies that
$$
x\in T_{\xi_0}\circ \cdots \circ T_{\xi_{N-1}} (I) \cap
T_{\omega_0}\circ \cdots \circ T_{\omega_{N-1}} (I),
$$
contradicting Claim~\ref{c.claimera52}. Thus we can assume that $i>0$ and 
define the cylinder 
$$
[\eta_{0}\ldots \eta_{(i-1)N-1}]\eqdef C_0 \ast C_1 \ast \cdots \ast C_{i-1}=
\widetilde C_0 \ast \widetilde C_1 \ast \cdots  \ast \widetilde C_{i-1}.
$$
Write $(i-1)N-1=r$.
By the definition of $\Si_x^j$ in \eqref{e.sigmaxj} we have 
\begin{equation}
\label{e.contradiction18abril}
x\in T_{\eta_0} \circ \cdots \circ T_{\eta_{r}}
\circ T_{\xi_0}\circ \cdots \circ T_{\xi_{N-1}} (I) \cap T_{\eta_0} \circ \cdots \circ T_{\eta_{r}}
\circ T_{\omega_0}\circ \cdots \circ T_{\omega_{N-1}} (I).
\end{equation}
Since for every $i$ we have that $T_{i}(J)\subset J$ and 
 $T_{i}|_{J}$ is injective, the intersection and union inclusion properties in Claim \ref{c.claimera52}
implies that 
$$ T_{\eta_0} \circ \cdots \circ T_{\eta_{r}}
\circ T_{\xi_0}\circ \cdots \circ T_{\xi_{N-1}} (I) \cap T_{\eta_0} \circ \cdots \circ T_{\eta_{r}}
\circ T_{\omega_0}\circ \cdots \circ T_{\omega_{N-1}} (I)=\emptyset,
$$
contradicting \eqref{e.contradiction18abril}.
Thus $C=\widetilde C$ and  proof of the claim is complete.
\end{proof}

To prove that
$\mathbb{P}^{+}(S_{x}^{j})\leq\mathbb{P}^{+}(Q^{j})$ note that
 $$
 \mathbb{P}^{+}(S_{x}^{j})
 \underset{(\textrm{a})}{=}\sum_{C\in\Sigma^{j}_{x}} \mathbb{P}^{+}(C)
  \underset{(\textrm{b})}{\leq}
 \sum_{C\in\Sigma^{j}_{x}} \mathbb{P}^{+}(F_{j}(C))
  \underset{(\textrm{c})}{=}\mathbb{P}^{+}\big( \bigcup_{C\in\Sigma^{j}_{x}} F_{j}(C)
  \big) 
   \underset{(\textrm{d})}{\le}  \mathbb{P}^+ (Q^j),
  $$ 
  where (a) follows from the disjointness of the cylinders $C\in \Sigma_x^j$,
  (b)
   from  Claim~\ref{c.substitutionmap1}, 
   (c)
 from the injectivity of $F_{j}$ (Claim~\ref{c.substitutionmap}), and 
  (d) from  $F_{j}(C)\in E_j \subset Q^{j}$.
The proof of the lemma is now complete. 
\end{proof}
This completes the proof of the proposition.
\end{proof}
The proof of Theorem~\ref{t.p.usadoAtdois} (i.e., $\mathbb{P}^{+}(\Sigma_{x})=0$) is now complete.
\end{proof}
The proof of Theorem~\ref{mt.abundance} is now complete.
\hfil \qed

%
%

\subsection{Probabilistic rigidity of $S_{\mathrm{t}}$}
\label{ss.prigidtyst}
In this section we see that under quite general conditions
the hypothesis 
$S_{\mathrm{t}}\cap \mathrm{supp}(\mathbb{P}^{+})\ne \emptyset$ implies that
 $\mathbb{P}^+(S_{\mathrm{t}})=1$. 
Recall the definition of the projection $\pi$ in \eqref{e.pi}.

\begin{teo}\label{t.rigidity}
Consider an injective
$\mathrm{IFS}(T_{1},\dots T_{k})$ defined on $I=[0,1]$. 
Let $\mathbb{P}^{+}$ be  a mixing Markov measure defined  on $\Sigma_k^+$ with
transition matrix $P=(p_{ij})$. 
\begin{itemize}
\item
If 
 there is $i\in\{1,\dots,k\}$ such that $\pi$ is not constant in 
 $ [i] \cap \mathrm{supp}(\mathbb{P}^{+})$ then $\mathbb{P}^{+}(S_{\mathrm{t}})=1$.
In particular, 
 $$
 \#\pi\big( S_{\mathrm{t}}\cap \mathrm{supp}(\mathbb{P}^{+}) \big)\geq k+1
 \quad \Longrightarrow \quad
\mathbb{P}^{+}(S_{\mathrm{t}})=1.
$$
 \item
If the maps $T_i$ have no common fixed points and for every  $i$ and $j$, with $i\neq j$, there is $m\in \{1,\dots, k\}$ with 
$p_{mi}\, p_{mj}>0$.
Then 
$$
S_{\mathrm{t}}\cap \mathrm{supp}(\mathbb{P}^{+})\neq \emptyset
\quad
\Longleftrightarrow
\quad
\mathbb{P}^{+} (S_{\mathrm{t}})=1.
$$
\end{itemize} 
\end{teo}

%

\begin{proof}
To prove the first item of the theorem
%
note that by hypothesis there is $i$ such that $\pi$ is not constant in 
 $ [i] \cap \mathrm{supp}(\mathbb{P}^{+})$.  Hence 
 $\xi, \omega \in  [i] \cap \mathrm{supp}(\mathbb{P}^{+})\cap S_{\mathrm{t}}$ such that $\pi(\xi)\neq \pi(\omega)$. 
 Thus there are $s$ and $\ell$ 
  such that 
 $$
 T_{\xi_0}\circ\dots\circ T_{\xi_{s}}
 (I)\cap T_{\omega_0}\circ\dots \circ T_{\omega_\ell}(I)=\emptyset.
 $$
 As $\xi,\omega\in [i] \cap \mathrm{supp}(\mathbb{P}^{+})$
 the cylinders
 $[\xi_0 \dots \xi_s]$ and $[\omega_0\dots \omega_\ell]$ are both admissible and satisfy
 $\xi_0=\omega_0=i$. This means that $\mathbb{P}^{+}$
 splits the IFS. Hence, by Theorem~\ref{mt.abundance}, 
  $\mathbb{P}^{+}(S_{\mathrm{t}})=1$.
 
 For the second part of the first item, just note that
 if $\#\pi\big( S_{\mathrm{t}}\cap \mathrm{supp}(\mathbb{P}^{+}) \big)\geq k+1$
 then from the pigeonhole principle there is $i$ such that $\pi$ is not constant in 
 $[i] \cap \mathrm{supp}(\mathbb{P}^{+})$. 

The 
implication 
$(\Leftarrow)$ in the
second item of the theorem is immediate. For the implication $(\Rightarrow)$
%
%
 we need the following lemma.
%
\begin{lema}\label{l.eraclaim}
For every  $\xi\in S_{\mathrm{t}}\cap \mathrm{supp}(\mathbb{P}^{+})$
there is $\omega \in  S_{t}\cap \mathrm{supp}(\mathbb{P}^{+})$ 
such that $\pi(\xi)\neq \pi(\omega)$.
\end{lema}
\begin{proof}
 Fix $\xi \in S_{\mathrm{t}}$. By definition of $S_{\mathrm{t}}$ we have that 
$$
\{\pi(\xi)\}= \bigcap_{n\geq 0} T_{\xi_{0}}\circ\cdots\circ T_{\xi_{n}}(I).
$$
 As the maps $T_{i}$ have no common fixed points 
 there is $i_{0}$ such that 
 $T_{i_{0}}(\pi(\xi))\neq \pi(\xi)$. The definition of an irreducible matrix implies that 
 there is an admissible cylinder of the form $[i_{0}i_{1}\dots i_{m} \xi_{0}]$. 
 Let 
 $$
 r\eqdef \mathrm{max}\big\{\ell\in \{0,\dots,m\} \colon T_{i_{\ell}}(\pi(\xi))\neq \pi(\xi)\big\} \ge 0.
 $$ 
  Consider the concatenation $\omega=i_{r}\dots i_{m}\ast \xi$. 
  Note  that, by definition, $\pi(\zeta)= T_{\zeta_0} 
  (\pi(\sigma (\zeta)))$ for  every $\zeta\in S_\mathrm{t}$. 
  Hence
  $$
  \pi(\omega)= T_{i_r}\circ \cdots \circ T_{i_m} (\pi(\xi)).
  $$
 By definition of $r$, 
 $$
 T_{i_m} (\pi(\xi))=\cdots = T_{i_{r+1}}(\pi(\xi))=\pi(\xi).
 $$
 Therefore 
 $$
 \pi(\omega) =T_{i_r}(\pi(\xi))\ne \pi(\xi).
 $$
  It remains to see that
  $\omega\in S_{\mathrm{t}}\cap \mathrm{supp}(\mathbb{P}^{+})$, for that just note that the cylinder $[i_0\dots i_m\xi_0]$ is admissible
  and $\xi \in \mathrm{supp}(\mathbb{P}^{+})$.
  This ends the proof of the lemma.
\end{proof}

Take sequences $\xi$ and $\omega$ as in Lemma~\ref{l.eraclaim}.
By definition of $\pi$, 
$$
\{\pi(\xi)\}= \bigcap_{n\geq 0} T_{\xi_{0}}\circ\cdots\circ T_{\xi_{n}}(I)
 \quad \textrm{and} \quad
\{\pi(\omega)\}=\bigcap_{n\geq 0} T_{\omega_{0}}\circ\cdots\circ T_{\omega_{n}}(I).
$$
As $\pi(\xi)\ne \pi(\omega)$  
 there are $\ell$ and $s$ such that 
\begin{equation}\label{e.sequences}
T_{\xi_{0}}\circ\cdots\circ T_{\xi_{\ell}}(I)\cap 
T_{\omega_{0}}\circ\cdots\circ T_{\omega_{s}}(I)=\emptyset.
\end{equation}
Note that the cylinders 
 $[\xi_0\dots \xi_\ell]$ and $[\omega_0\dots\omega_s]$ are admissible.
If $\xi_0=\omega_0$ we are done. Otherwise,  $\xi_0\ne \omega_0$ and by hypothesis there is $m$ such that $p_{m\xi_0}>0$ and $p_{m\omega_0}>0$.
This implies that the cylinders $[m\xi_0\dots \xi_\ell]$ and $[m\omega_0\dots\omega_s]$  are 
both admissible. Since the 
 maps $T_{i}$ are injective it follows from \eqref{e.sequences}
$$
 T_m\circ
T_{\xi_{0}}\circ\cdots\circ T_{\xi_{\ell}}(I)\cap 
T_m\circ T_{\omega_{0}}\circ\cdots\circ T_{\omega_{s}}(I)=\emptyset.
$$
Therefore $\mathbb{P}^+$ splits the IFS and by Theorem~\ref{mt.abundance} we have
$\mathbb{P}^{+} (S_{\mathrm{t}})=1$. This ends the proof of the theorem.
 \end{proof}

 \subsection{Separability}\label{ss.separability}
In this section we give  some characterisations of a
separable IFS. Note that item (2) in the next theorem means that the IFS is separable.

%
\begin{teo}\label{t.separableee} 
Consider an  $\mathrm{IFS}(T_{1},\dots T_{k})$ defined on $I=[0,1]$.
Suppose that there is some non-trivial closed
interval $J$ 
such that $T_i(J)\subset J$ and $T_i|_{J}$ is injective for every $j\in \{1,\dots,k\}$.
Then  the following assertions are equivalent:
 \begin{enumerate}
  \item
 The maps of the IFS have no common fixed points and $S_{\mathrm{t}}\neq\emptyset$.
 \item
 The target set $A_{\mathrm{t}}$ has at least two elements.
 \item
 There are finite sequences $\xi_{1}\dots{\xi_{\ell}}$ and $\omega_{1}\dots \omega_{s}$ 
 such that 
\[
\begin{split}
&T_{\xi_{1}}\circ\dots \circ T_{\xi_{\ell}}(I)\cap T_{\omega_{1}}\circ\dots\circ T_{\omega_{s}}(I)=\emptyset
\quad \mbox{and}\\
&T_{\xi_{1}}\circ\dots \circ T_{\xi_{\ell}}(I)\cup T_{\omega_{1}}\circ\dots\circ T_{\omega_{s}}(I)\subset J.
\end{split}
\]
\item
 The maps of the IFS have no common fixed point and $\mathbb{P}^{+}(S_{\mathrm{t}})=1$ for every 
 mixing
 Markov measure $\mathbb{P}^{+}$ whose support is the whole $\Sigma_k^+$.
\end{enumerate}
\end{teo}

\begin{proof}
To prove the implication $(1)\Rightarrow (2)$ note that
 since $S_{\mathrm{t}}\neq \emptyset$ there is  $p\in A_{\mathrm{t}}$.
 Since the maps of the IFS have no common fixed point there is
$i$ such that 
 $T_{i}(p)\neq p$. The invariance of $A_{\mathrm{t}}$ implies that $T_{i}(p)\in A_{\mathrm{t}}$.
 Thus $\{p, T_i(p)\} \subset  A_{\mathrm{t}}$
 and we are done.
 
To see that
$(2)\Rightarrow (3)$ we need the following claim:
\begin{claim}
 $\#(A_{\mathrm{t}}\cap \mathrm{int}(J))\ge 2.$ 
\end{claim}
\begin{proof}
Since $T_i(J)\subset J$ for every $i$ we have that $\mathcal{B}(J)\subset J$. Hence Propositions \ref{p.existence} and \ref{p.l.minimality} implies that 
 $A_{t}\subset J$. The claim follows from Proposition~\ref{p.dichotomy}.
\end{proof}

Take  two different points $p,q\in A_{\mathrm{t}}\cap \mbox{int}\, J$  and consider disjoint
  neighbourhoods $U$ and $V$ of $p$ and $q$, respectively, such that $U\cup V\subset J$. By the definition of $ A_{\mathrm{t}}$  there are sequences $\xi$ and $\omega$ such that 
  $$
  \{p\}=\bigcap_{n\geq 0} T_{\xi_{0}}\circ \dots \circ T_{\xi_{n}}(I) \quad
  \mbox{and}\quad \{q\}=\bigcap_{n\geq 0} T_{\omega_{0}}\circ \dots \circ T_{\omega_{n}}(I). 
  $$
 Hence there are $n_{0}$ and $m_{0}$ such that 
$T_{\xi_{0}}\circ \dots \circ T_{\xi_{n_{0}}}(I)\subset U$
and $T_{\omega_{0}}\circ \dots \circ T_{\omega_{m_{0}}}(I)\subset V$.
Since $U\cap V=\emptyset$ we get the implication $(2)\Rightarrow (3)$.

To prove
$(3)\Rightarrow (4)$ consider the
finite sequences $\xi_{1}\dots{\xi_{\ell}}$ e $\omega_{1}\dots \omega_{s}$ in item (3). 
Clearly the condition in (3)  prevents the existence of a common fixed point. On the other hand,
since $T_1(J)\subset J$ and $T_1|_{J}$ is injective, we have that 
\[
\begin{split}
&T_{1}\circ T_{\xi_{1}}\circ\dots \circ T_{\xi_{\ell}}(I)\cap T_{1}\circ T_{\omega_{1}}\circ\dots\circ T_{\omega_{s}}(I)=\emptyset
\quad 
\mbox{and}\\
&T_{1}\circ T_{\xi_{1}}\circ\dots \circ T_{\xi_{\ell}}(I)\cup T_{1}\circ T_{\omega_{1}}\circ\dots\circ T_{\omega_{s}}(I)\subset J.
\end{split}
\]
Thus every mixing Markov measure with full support $\mathbb{P}^{+}$ splits the IFS in $J$. Now
Theorem \ref{mt.abundance} implies  that $\mathbb{P}^{+}(S_{\mathrm{t}})=1$ and we are done.

The implication $(4)\Rightarrow (1)$ is immediate.
\end{proof}



\section{Asymptotic stability on measures}\label{s.stationary}


In this section we prove Theorems~\ref{mt.eimportante} and 
\ref{mt.generalizedB} in Sections~\ref{ss.stationaryyy} and 
\ref{ss.fimdehora}, respectively.

\subsection{Stationary measures for IFSs with probabilities in $[0,1]$}
\label{ss.stationaryyy}
In this section we prove Theorem~\ref{mt.eimportante}. For that
we consider a separable $\mathrm{IFS}(T_{1},\dots T_{k};\mathfrak{b})$
defined
on $I=[0,1]$, its Markov operator $\mathfrak{T}=\mathfrak{T}_\mfb$, and
its coding map $\pi$ in \eqref{e.pi}, we see that for every probability measure $\mu\in \mathcal{M}_{1}(I)$ 
it holds 
 $$
 \lim_{n\to \infty} \mathfrak{T}^{n}\mu=\pi_{*}\mathfrak{b}
\quad \mbox{(asymptotic stability)}.
 $$
 The main step of the proof of the theorem is the 
 next proposition  that states a sufficient condition for 
 the asymptotic stability of
 an IFS with 
 probabilities.

  \begin{prop}
\label{p.attractingstationary}
Consider an $\mathrm{IFS}(T_{1},\dots T_{k};\mfb)$ with probabilities 
defined on a compact metric space $X$. 
Suppose that $\mathfrak{b}(S_{\mathrm{t}})=1$. Then for every 
probability measure $\mu\in \mathcal{M}_{1}(X)$ 
it holds 
 $$
 \lim_{n\to \infty} \mathfrak{T}_{\mfb}^{n}\mu=\pi_{*}\mathfrak{b}.
 $$
In particular, $\mu_{\mathfrak{b}}\eqdef\pi_{*}\mathfrak{b}$ is the unique stationary measure of 
$\mathrm{IFS}(T_{1},\dots T_{k};\mfb)$. 
Furthermore, $\mathrm{supp}(\mu_{\mathfrak{b}})=\overline{A_{\mathrm{t}}}$.
\end{prop}

%

%

We postpone the proof of Proposition~\ref{p.attractingstationary} and deduce the theorem from it.
 
 \subsubsection{Proof of Theorem~\ref{mt.eimportante}}
 In view of Proposition \ref{p.attractingstationary} it is sufficient to prove that
 $\mathfrak{b}(S_{\mathrm{t}})=1$ and  the measure $\pi_{*}\mathfrak{b}$ is continuous. Since the IFS is 
 separable and every Bernoulli measure (with strictly positive weights) is a mixing Markov measure,
 Theorem \ref{t.separableee} implies that $\mathfrak{b}(S_{\mathrm{t}})=1$.
To see that $\pi_{*}\mathfrak{b}$ is continuous  we need to prove that  $\pi_{*}\mathfrak{b}(\{x\})=0$
for every $x\in [0,1]$. Take $x\in [0,1]$ and recall the definition of the set 
$\Sigma_{x}$ in \eqref{e.sigmax}. 
Since $\pi^{-1}(x)\subset \Sigma_{x}$ we have that 
$$
\pi_{*}\mathfrak{b}(\{x\})=\mathfrak{b}(\pi^{-1}(x))\leq \mathfrak{b}(\Sigma_{x})=0,
$$
where the last equality follows from Theorem \ref{t.p.usadoAtdois}.
The proof of Theorem \ref{mt.eimportante} is now complete. 
\hfil \qed  
%

%

\subsubsection{Proof of Proposition \ref{p.attractingstationary}}
We assume that $\mfb=\mfb(p_1, \dots,p_k)$ and write 
$\mathfrak{T}=\mathfrak{T}_{\mfb}$.
We begin by proving two auxiliary lemmas:

\begin{lema}\label{l.BHstationary}
For every stationary measure of $\mathfrak{T}$
 it holds $\mathcal{B}(\mathrm{supp}(\mu))\subset\mathrm{supp}(\mu)$.
\end{lema}

\begin{proof}
It is sufficient to show that 
$T_{i}(\mathrm{supp}(\mu))\subset\mathrm{supp}(\mu)$ for every $i$. 
Given $x\in\mathrm{supp}(\mu)$ take a neighborhood $V$ of $T_{i}(x)$. By the choice
of $x$, 
$\mu(T_{i}^{-1}(V))>0$.
Since $\mu$ is a stationary measure we have 
$$
\mu(V)=p_{1}\mu(T_{1}^{-1}(V))+\cdots
+p_{k}\mu(T_{k}^{-1}(V))\geq p_{i}\mu(T_{i}^{-1}(V))>0,
$$
 proving the lemma. 
 \end{proof}

\begin{lema}\label{l.interessante}
Consider the $\mathrm{IFS}(T_{1},\dots T_{k})$.
Then for every sequence $(\mu_{n})$ of probabilities of $\mathcal{M}_{1}(X)$  and every
 $\omega\in S_{\mathrm{t}}$ it holds
$$
\lim_{n\to \infty}T_{\omega_{0}*}\ldots_\ast T_{\omega_{n}*}\mu_{n}=\delta_{\pi(\omega)}.
$$
\end{lema}

\begin{proof}
Consider a sequence of probabilities $(\mu_n)$ and  $\omega\in S_{\mathrm{t}}$. Fix any 
$g\in C^{0}(X)$. Then  given any $\epsilon>0$ there is $\delta$ 
such that 
$$
|g(y)-g\circ\pi(\omega)|<\epsilon
\quad \mbox{for all $y\in X$ with $d(y,\pi(\omega))<\delta$.}
$$
Since $\omega \in S_{\mathrm{t}}$ there is $n_{0}$
such that  $d(T_{\omega_{0}}\circ\cdots \circ T_{\omega_{n}}(x),\pi(\omega))< \delta$ for every
$x\in X$ and every $n\geq n_{0}$. Therefore 
for $n\geq n_{0}$ we have
\[
\begin{split}
\left|
g\circ \pi(\omega)-
\int g\, dT_{\omega_{0}*}\ldots_\ast T_{\omega_{n}*}\mu_{n}\right|&=
\left|
\int 
g\circ \pi(\omega)\,d\mu_{n}-
\int g\circ T_{\omega_{0}}\circ\cdots\circ T_{\omega_{n}}(x)\,d \mu_{n}
\right|\\
&\leq \int 
|g\circ \pi(\omega)- g\circ T_{\omega_{0}}\circ\cdots\circ T_{\omega_{n}}(x)|\, d\mu_{n}
\le \epsilon.
\end{split}
\]
This implies that
$$
\lim_{n\to \infty}\int g\, dT_{\omega_{0}*}\ldots_\ast T_{\omega_{n}*}\mu_{n}=g\circ\pi(\omega)
$$
Since this holds for every continuous map $g$ the lemma follows.
\end{proof}

We will show that $
\lim_{n\to\infty}\mathfrak{T}^{n}\nu=\pi_{*}\mathfrak{b}
$
for every $\nu \in \mathcal{M}_{1}(X)$.
 In particular,  by the 
continuity of $\mathfrak{T}$, $\mathfrak{T}\pi_{*}\mathfrak{b}=\pi_{*}\mathfrak{b}$.

Note that from  the definition of the Markov operator in \eqref{e.markov}, for  every $\nu\in \mathcal{M}_{1}(X)$ and every continuous map $f\in C^{0}(X)$ 
it holds 
\begin{equation}
 \label{e.operatormarkov}
\int fd(\mathfrak{T}^{n}\nu)
=\sum_{\xi_0,\dots ,\xi_{n-1}}p_{\xi_{0}}p_{\xi_{1}}\ldots p_{\xi_{n-1}}\, 
\int  f \, d T_{\xi_{0}*} T_{\xi_{1}*} \ldots _\ast T_{\xi_{n-1}*}\nu.
\end{equation}

Fixed $\nu \in \mathcal{M}_{1}(X)$ 
consider the sequence of functions $F_{n}:\Sigma_{k}^+\rightarrow \mathds{R}$ defined
 by
$$
F_{n}(\xi) \eqdef\displaystyle\int f \, d T_{\xi_{0}*} T_{\xi_{1}*} \ldots _\ast T_{\xi_{n-1}*}\nu. 
$$
Since the map
$F_{n}$  is constant in the cylinders $[\xi_{0},\ldots,\xi_{n-1}]$,
it is a measurable function.
From this property,  
 equation~\eqref{e.operatormarkov}, and the definition of the Bernoulli measure $\mathfrak{b}$ we have
$$
\int f\, d(\mathfrak{T}^{n}\nu)=
\int F_{n}\, d\mathfrak{b}.
$$

By hypothesis $\mathfrak{b}(S_{\mathrm{t}})=1$, thus applying Lemma~\ref{l.interessante} 
to the constant sequence $\mu_{n}=\nu$  we have that 
\begin{equation}\label{e.limitFn}
\lim_{n\rightarrow\infty} F_{n}(\xi)=f\circ \pi(\xi)
\quad 
\mbox{for  $\mathfrak{b}$-a.e. $\xi$}.
\end{equation}
Since $|F_{n} (\xi)|\leq \|f\|$, from \eqref{e.limitFn} using the dominated convergence theorem 
we get
$$
\lim_{n\rightarrow \infty}\int f\, d(\mathfrak{T}^{n}\nu)
=\lim_{n\rightarrow\infty}\int F_{n}\, d\mathfrak{b}
=\int f\circ\pi \, d\mathfrak{b}
=\int fd\pi_{*}\mathfrak{b}.
$$
Since the previous equality holds for every continuous map $f$ 
it follows  that $\pi_\ast \mfb$ is an attracting measure.

It remains to see that $\mathrm{supp}(\pi_{*}\mathfrak{b})=\overline{A_{\mathrm{t}}}$.
For that note the following equalities
\begin{equation}\label{e.characterisationestationaty}
\pi_{*}\mathfrak{b}(A_\mathrm{t})= \mathfrak{b}(\pi^{-1} (A_\mathrm{t})) =
\mathfrak{b}(S_\mathrm{t})=1
\end{equation}
that imply $\mathrm{supp} (\pi_{*}\mathfrak{b})\subset \overline{A_\mathrm{t}}$. 

To get $\mathrm{supp} (\pi_{*}\mathfrak{b})\supset \overline{A_\mathrm{t}}$
recall  that, by  Proposition~\ref{p.existence},
every $\mathcal{B}$-invariant compact set contains a fixed point of 
$\mathcal{B}$.
By Lemma~\ref{l.BHstationary}
we have
$\mathcal{B}(\mathrm{supp}(\pi_{*}\mathfrak{b}))\subset \mathrm{supp}(\pi_{*}\mathfrak{b})$.
Hence $\mathrm{supp}(\pi_{*}\mathfrak{b}))$ contains a fixed point of $\mathcal{B}$.
As $\overline{A_\mathrm{t}}$ is a minimum fixed point of $\mathcal{B}$ (see 
Proposition~\ref{p.l.minimality})
this implies that
$\overline{A_\mathrm{t}}\subset \mathrm{supp}(\pi_{*}\mathfrak{b})$. 
Thus $\mathrm{supp}(\pi_{*}\mathfrak{b})=\overline{A_{\mathrm{t}}}$, completing the proof of the proposition.
\hfil \qed

\smallskip

The previous proposition provides a (unique) stationary measure whose
 support is  $\overline{A_{\mathrm{t}}}$. To prove that the support of
 this measure is the closure
 of the target  we use the characterisation of the stationary measure
 in \eqref{e.characterisationestationaty}.
 Next proposition
 claims that the support of the stationary measure of an asymptotically stable Markov 
 operator of an IFS with $S_{\mathrm{t}}\neq \emptyset$ always is
 $\overline{A_{\mathrm{t}}}$, even when
   $\mathfrak{b}(S_{\mathrm{t}})=0$ (recall 
 that either $\mathfrak{b}(S_{\mathrm{t}})=1$ or $\mathfrak{b}(S_{\mathrm{t}})=0$). 

\begin{prop}\label{p.ultimosdias}
Consider an $\mathrm{IFS}(T_{1},\dots,T_{k};\mathfrak{b})$ with 
probabilities defined on a compact 
metric space whose Markov 
operator  $\mathfrak{T}_\mfb$
is asymptotically stable and let $\mu$ be  its stationary measure.
If $S_{\mathrm{t}}\neq \emptyset$ then
$\mathrm{supp}({\mu})=\overline{A_{\mathrm{t}}}$.
\end{prop}
\begin{proof}
The inclusion $\mathrm{supp}({\mu})\supset \overline{A_{\mathrm{t}}}$
follows from Lemma \ref{l.BHstationary}. 
To prove the inclusion  ``$\subset$'' take any point $p\in \mathrm{supp}({\mu})$ and an 
open neighbourhood $V$ of $p$. We need to see that $V\cap A_{\mathrm{t}}\neq\emptyset$.
For this take any point $x\in A_{\mathrm{t}}$. Since $\mathfrak{T}=\mathfrak{T}_\mfb$
is asymptotically stable 
Alexandrov's theorem (see \cite[page 60]{Bara})
 implies that 
$$
\liminf_{n} \mathfrak{T}^{n}\delta_{x}(V)\geq \mu(V)>0.
$$
Hence there is $n_{0}$ such that $\mathfrak{T}^{n_{0}}\delta_{x}(V)>0$.
By definition of the Markov operator we have that 
$$
\mathfrak{T}^{n_{0}}\delta_{x}(V)=\sum_{\xi_0,\dots ,\xi_{n_{0}-1}}p_{\xi_{0}}p_{\xi_{1}}\ldots p_{\xi_{n_{0}-1}}\, 
 T_{\xi_{0}*} T_{\xi_{1}*} \ldots _\ast T_{\xi_{n_{0}-1}*}\delta_{x}(V).
$$
Therefore there is a finite sequence $\xi_0\dots \xi_{n_{0}-1}$ such that 
$$
\delta_{x}(T_{\xi_{n_{0}-1}}^{-1}\circ \dots \circ T_{\xi_{0}}^{-1}(V))>0
$$
and thus
 $x\in T_{\xi_{n_{0}-1}}^{-1}\circ \dots \circ T_{\xi_{0}}^{-1}(V)$.
The invariance of $A_{\mathrm{t}}$ now implies that $V\cap A_{\mathrm{t}}\neq \emptyset$,
proving the proposition.
\end{proof}

\subsection{Stationary measures for recurrent IFSs in $[0,1]$}
\label{ss.fimdehora}

 In this section
 we prove Theorem~\ref{mt.generalizedB}.
 For that we consider a recurrent $\mathrm{IFS}(T_{1},\dots T_{k};\mathbb{P}^{+})$
defined on a compact metric space $X$, where  
$\mathbb{P}^+$ is the Markov probability 
associated to $(P=(p_{i,j}), \bar p=(p_i))$.
We also consider the set $\widehat{X}= X \times \{1,\dots,k\}$ and
the  (generalised) Markov operator $\mathfrak{S}=\mathfrak{S}_{\mathbb{P}^+}$ (see 
\eqref{e.generalised}) and the 
generalised coding map 
$\varpi \colon S_{\mathrm{t}} \to \widehat{X}$ given by  $\varpi(\xi)\eqdef(\pi(\xi),\xi_{0})$
 (see \eqref{e.generalisedcodingmap}) of the IFS. A final ingredient is the
 inverse Markov measure $\mathbb{P}^{-}$ associated to $\mathbb{P}^{+}$
defined in \eqref{e.markovinverse}.
 
To prove Theorem~\ref{mt.generalizedB} we need to see that
every
 $\mathrm{IFS}(T_{1},\dots T_{k};\mathbb{P}^{+})$ 
 such that
the
inverse Markov measure $\mathbb{P}^{-}$
 is  mixing and splits the $\mathrm{IFS}$
in some non-trivial closed interval $J$ satisfies
$$
\lim_{n\to \infty}\mathfrak{S}^{n}(\hat{\mu})=\varpi_{*}\mathbb{P}^{-}
\quad
\mbox{for every
$\hat{\mu}\in \mathcal{M}_{1}([0,1]\times \{1,\dots,k\})$.}
$$
The main step of the proof of Theorem~\ref{mt.generalizedB} is the following result whose proof is postponed.

\begin{teo}\label{t.p.t.recurrent}
Consider a recurrent  $\mathrm{IFS}(T_{1},\dots T_{k};\mathbb{P}^{+})$ defined on a compact 
metric space $X$ such that
 $\mathbb{P}^{-}$ is mixing and 
 $\mathbb{P}^{-}(S_{\mathrm{t}})=1$.
Then
$$
\lim_{n\to \infty}\mathfrak{S}^{n}(\widehat{\mu})=\varpi_{*}\mathbb{P}^{-}
\quad
\mbox{for every
$\widehat{\mu}\in \mathcal{M}_{1}(\widehat{X})$.}
$$
In particular, $\varpi_{*}\mathbb{P}^{-}$ is the unique stationary measure of $\mathfrak{S}$.
%
\end{teo}

\begin{proof}[Proof of Theorem \ref{mt.generalizedB}]
Since $\mathbb{P}^{-}$ is mixing and splits the $\mathrm{IFS}$ in some non-trivial interval
it follows from Theorem~\ref{mt.abundance}
that
$\mathbb{P}^{-}(S_{\mathrm{t}})=1$.
 Thus the
theorem follows from
 Theorem~\ref{t.p.t.recurrent}. 
 \end{proof}

 \begin{proof}[Proof of Theorem \ref{t.p.t.recurrent}]  
Given a function  $\widehat f\colon \widehat{X}\to \mathbb{R}$
we define its $i$-section
$f_i\colon X \to \mathbb{R}$
 by 
$f_{i}(x)\eqdef \widehat f(x,i)$ and write $\widehat f= \langle f_{1},\dots,f_{k} \rangle$.  
We need to see that for every 
measure
$\widehat{\mu}=(\mu_{1},\dots,\mu_{k})\in \mathcal{M}_{1}(\widehat{X})$
 and every continuous 
 function $\widehat{f}=\langle f_{1},\dots,f_{k}\rangle\in C^{0}(\widehat{X} )$ it holds 
 \begin{equation}\label{e.atra}
\lim_{n\to \infty}\int \widehat{f}\, d\mathfrak{S}^{n}({\widehat{\mu}})=\int \widehat{f}  \, d\varpi_{*}\mathbb{P}^{-}.
  \end{equation}

By definition, it follows that
$$
\int \widehat f\, d\widehat\mu=\sum_{i=1}^{k} \int f_{i}\, d\mu_{i},
\quad 
\mbox{where} 
\quad
\widehat{\mu}=(\mu_{1},\dots,\mu_{k}),
$$
and hence
 \begin{equation}\label{e.3101}
\int \widehat f\, d\mathfrak{S}^{n}({\widehat{\mu}})=\sum_{j=1}^{k}\int f_{j}
\,d  (\mathfrak{S}^{n}({\widehat{\mu}}))_{j},
\quad 
\mathfrak{S}^{n}({\widehat{\mu}})=\big( (\mathfrak{S}^{n}({\widehat{\mu}}))_{1}, \dots, (\mathfrak{S}^{n}({\widehat{\mu}}))_{k}\big).
 \end{equation}


To get the convergence of the integrals of the sum in  \eqref{e.3101}  
we need a preparatory lemma. First,
denote by $\Vert g\Vert $ the uniform
 norm of a continuous function $g\colon X \to \mathbb{R}$.
\begin{lema}\label{l.primitiveee}
Consider $\widehat{\mu}=(\mu_{1},\dots,\mu_{k})\in \mathcal{M}_{1}(\widehat{X})$ such that 
$\mu_{i}(X)>0$ for every $i\in \{1,\dots, k\}$. Then for every $g\in C^{0}(X)$ it holds
$$
\limsup_{n}\left|\int g \, d (\mathfrak{S}^{n}({\widehat{\mu}}))_{j}
-\int_{[j]} (g\circ\pi)\, d\,\mathbb{P}^{-} \right|\leq k\, \Vert g\Vert\, \max_{i}|\mu_{i}(X)-p_i|,
$$
where $\bar p=(p_1,\dots,p_k)$ is the unique stationary vector of $P$.
\end{lema}
\begin{proof}
Take $\widehat{\mu}\in\mathcal{M}_{1}(\widehat{X})$ 
as in the statement of the lemma
and for each $i$ define
the
probability measure $\overline{\mu}_{i}$ 
$$
\overline{\mu}_{i}(B)\eqdef \frac{\mu_{i}(B)}{\mu_{i}(X)},
\quad \mbox{where $B$ is a Borel subset of $X$}.
$$
A straightforward calculation and the previous definition imply that  
\[
\begin{split}
 (\mathfrak{S}^{n}({\widehat{\mu}}))_{j}&=\sum_{\xi_1,\dots ,\xi_{n}}
 p_{\xi_{n}\xi_{n-1}}\dots p_{\xi_{2}\xi_{1}}p_{\xi_{1}j} \, T_{j*}T_{\xi_{1}*}
 \dots_{*}T_{\xi_{n-1}*}\mu_{\xi_{n}}\\    
&=
 \sum_{\xi_1,\dots ,\xi_{n}}
 \mu_{\xi_{n}}(X)\, p_{\xi_{n}\xi_{n-1}}\dots p_{\xi_{2}\xi_{1}}p_{\xi_{1}j}
  T_{j*}T_{\xi_{1}*}\dots_{*}T_{\xi_{n-1}*}\overline{\mu}_{\xi_{n}}. 
 \end{split}
 \]
  Thus given any $g\in C^{0}(X)$ we have that
 \[ \int g\, d(\mathfrak{S}^{n}({\widehat{\mu}}))_{j}= \sum_{\xi_1,\dots ,\xi_{n}}
 \mu_{\xi_{n}}(X)\, p_{\xi_{n}\xi_{n-1}}\dots 
 p_{\xi_{1}j}
  \int g \,d T_{j*}T_{\xi_{1}*}\dots_{*}T_{\xi_{n-1}*}\overline{\mu}_{\xi_{n}}.
\]
%
  
  Let 
  $$
  L_{n}\eqdef\left|\int g \, d (\mathfrak{S}^{n}({\widehat{\mu}}))_{j}
-\int_{[j]} (g\circ\pi)\, d\,\mathbb{P}^{-} \right|
$$
 and write  $ \mu_{\xi_{n}}(X)=
 (\mu_{\xi_{n}}(X)-p_{\xi_{n}})+ p_{\xi_{n}}$. Then

\[
\begin{split}
L_{n}&\leq \left| \sum_{\xi_1,\dots ,\xi_{n}}
 (\mu_{\xi_{n}}(X)-p_{\xi_{n}})\, p_{\xi_{n}\xi_{n-1}}\dots 
 p_{\xi_{1}j}
  \int g \,d T_{j*}T_{\xi_{1}*}\dots_{*}T_{\xi_{n-1}*}\overline{\mu}_{\xi_{n}}\right|\\
  & + \left|\sum_{\xi_1,\dots ,\xi_{n}}
 p_{\xi_{n}}\, p_{\xi_{n}\xi_{n-1}}\dots 
 p_{\xi_{1}j}
  \int g \,d T_{j*}T_{\xi_{1}*}\dots_{*}T_{\xi_{n-1}*}\overline{\mu}_{\xi_{n}}-
  \int_{[j]} (g\circ\pi)\, d\,\mathbb{P}^{-}\right|\\
  &\leq \max_{i}|\mu_{i}(X)-p_i|
\Vert g\Vert\sum_{\xi_1,\dots ,\xi_{n}}
 p_{\xi_{n}\xi_{n-1}}\dots 
 p_{\xi_{1}j}\\
 & +
  \left|\sum_{\xi_1,\dots ,\xi_{n}}
 p_{\xi_{n}}\, p_{\xi_{n}\xi_{n-1}}\dots 
 p_{\xi_{1}j}
  \int g \,d T_{j*}T_{\xi_{1}*}\dots_{*}T_{\xi_{n-1}*}\overline{\mu}_{\xi_{n}}-
  \int_{[j]} (g\circ\pi)\, d\,\mathbb{P}^{-}\right|.\\
\end{split}
\]
Note that $\sum_{\xi_1,\dots ,\xi_{n-1}}
 p_{\xi_{n}\xi_{n-1}}\dots
 p_{\xi_{1}j}$ is the entry $(\xi_{n},j)$ of the matrix $P^{n}$. Hence
 $$
 \sum_{\xi_1,\dots ,\xi_{n}}
 p_{\xi_{n}\xi_{n-1}}\dots 
 p_{\xi_{1}j}=\sum_{\xi_{n}=1}^k \, \sum_{\xi_1,\dots ,\xi_{n-1}}
 p_{\xi_{n}\xi_{n-1}}\dots
 p_{\xi_{1}j}\leq k.
  $$
  Therefore
  \begin{equation} \label{e.sumsum}
   \max_{i}|\mu_{i}(X)-p_i|
\Vert g\Vert\sum_{\xi_1,\dots ,\xi_{n}}
 p_{\xi_{n}\xi_{n-1}}\dots 
 p_{\xi_{1}j} \le 
 k\, \Vert g\Vert\, \max_{i}|\mu_{i}(X)-p_i|.
  \end{equation}
 
 We now estimate the second parcel in the sum above. 
 \begin{claim}\label{c.segunda25} For every continuous function $g$ it holds
 $$
 \lim_{n\to \infty}\sum_{\xi_1,\dots ,\xi_{n}}
 p_{\xi_{n}}p_{\xi_{n}\xi_{n-1}}\dots 
 p_{\xi_{1}j}
  \int g \,d T_{j*}T_{\xi_{1}*}\dots_{*}T_{\xi_{n-1}*}\overline{\mu}_{\xi_{n}}
  =\int_{[j]} (g\circ\pi)\, d\,\mathbb{P}^{-}.$$
 \end{claim}
 
Observe that equation \eqref{e.sumsum} and the claim imply the
 lemma.

 \begin{proof}[Proof of Claim \ref{c.segunda25}]
 Consider the sequence of functions 
given by
$$
G_{n}:\Sigma_{k}^+\rightarrow \mathbb{R}, \quad 
G_{n}(\xi) \eqdef\displaystyle\int g\,d T_{\xi_{0}*}T_{\xi_{1}*}\dots_{*}T_{\xi_{n-1}*}\overline{\mu}_{\xi_{n}}.
$$
By definition,
for every $n$ the corresponding
 map
$G_{n}$ is constant in the cylinders $[\xi_{0},\ldots,\xi_{n}]$ and thus
it is  measurable. By definition of $\mathbb{P}^{\pm}$, for every $j$ we have that 
$$
p_{\xi_{n}}p_{\xi_{n}\xi_{n-1}}\dots p_{\xi_{2}\xi_{1}}p_{\xi_{1}j}
=\mathbb{P}^{+}([\xi_{n}\xi_{n-1}\dots\xi_{1}j])=\mathbb{P}^{-}([j\xi_{1}\xi_{2}\dots\xi_{n}]).
$$
Hence 
$$
 \sum_{\xi_1,\dots ,\xi_{n}}
 p_{\xi_{n}}p_{\xi_{n}\xi_{n-1}}\dots 
 p_{\xi_{1}j}
  \int g \,d T_{j*}T_{\xi_{1}*}\dots_{*}T_{\xi_{n-1}*}\overline{\mu}_{\xi_{n}}=\int_{[j]} G_{n}\,d\,\mathbb{P}^{-}.
$$
It follows from  the hypothesis $\mathbb{P}^{-}(S_{\mathrm{t}})=1$ and Lemma \ref{l.interessante} that 
\begin{equation}\label{e.limitFn2}
\lim_{n\rightarrow\infty} G_{n}(\xi)=g\circ \pi(\xi)
\quad 
\mbox{for  $\mathbb{P}^{-}$-almost every  $\xi$}.
\end{equation}

  Now note that $|G_{n} (\xi)|\leq \|g\|$ for every $\xi\in \Sigma_{k}^{+}$. 
From \eqref{e.limitFn2}, using the dominated convergence theorem, 
we get
$$
\lim_{n\rightarrow\infty}\int_{[j]} G_{n}\, d\,\mathbb{P}^{-}
=\int_{[j]} (g\circ\pi)\, d\,\mathbb{P}^{-},
$$
ending the proof of the claim.
 \end{proof}
%
%
The proof of the lemma is now complete.
 \end{proof}

 To prove the theorem observe that  since $\mathbb{P}^-$ is mixing the transition  matrix $P$ 
 associated to $\mathbb{P}^+$ 
 is primitive, recall Section~\ref{sss.inversemarkov}.
 Take $\widehat{\mu}=(\mu_1,\dots, \mu_k)\in \mathcal{M}_{1}(\widehat{X})$. 
 Note that by definition of the Markov operator
  $$
  \big( (\mathfrak{S}\widehat{\mu})_{1}(X),\dots,(\mathfrak{S}\widehat{\mu})_{k}(X)\big)=
 \widehat{p}\,P, \quad
 \mbox{ where $\widehat{p}=(\mu_{1}(X),\dots,\mu_{k}(X))$}.
 $$
  Hence for every $n\geq 1$
 \begin{equation}\label{e.perron}
  \big( (\mathfrak{S}^{n}\widehat{\mu})_{1}(X),\dots,(\mathfrak{S}^{n}\widehat{\mu})_{k}(X) \big)=
 \widehat{p}\,P^{n}.
  \end{equation}
 
By the Perron-Frobenius theorem, see for instance \cite[page 64]{Mane}, we have that $P$ 
the stationary vector $\bar p=(p_1,\dots, p_k)$ 
is positive\footnote{A vector
$v=(v_{1},\dots,v_{k})$ is said positive if $v_{i}>0$ for all $i$. } 
and
$$
\lim_{n\to \infty}\widehat{p}\,P^{n}=\bar{p}
\quad
\mbox{for every probability vector $\widehat{p}$}.
$$
Hence \eqref{e.perron}  gives 
$n_{0}$ such that 
the vector $\big((\mathfrak{S}^{n_{1}}\widehat{\mu})_{1}(X),\dots,(\mathfrak{S}^{n_{1}}\widehat{\mu})_{k}(X)\big)$
 is positive for every $n_{1}\geq n_{0}$. Therefore we can apply 
 Lemma \ref{l.primitiveee} to the measure $\mathfrak{S}^{n_{1}}({\widehat{\mu}})$ 
 for every $n_{1}\geq n_{0}$, 
 obtaining for every $g\in C^{0}(X)$ the inequality
$$
\limsup_{n}\left|\int g \, d (\mathfrak{S}^{n+n_{1}}({\widehat{\mu}}))_{j}
-\int_{[j]} (g\circ\pi)\, d\,\mathbb{P}^{-} \right|\leq k\, \Vert g\Vert\, \max_{i}
|(\mathfrak{S}^{n_{1}}\widehat{\mu})_{i}(X)-p_i|.
$$
 It follows from the definition of $\limsup$ and  
 the previous inequality  that 
$$
 \limsup_{n}\left|\int g \, d (\mathfrak{S}^{n}({\widehat{\mu}}))_{j}
-\int_{[j]} (g\circ\pi)\, d\,\mathbb{P}^{-} \right|\leq k
\, \Vert g\Vert \,
\max_{i}
|(\mathfrak{S}^{n_{1}}\widehat{\mu})_{i}(X)-p_i|
 $$
 for every $n_{1}\geq n_{0}$. By \eqref{e.perron} and 
 the Perron-Frobenius theorem we get
 $$\lim_{n_{1}\to \infty}\max_{i}
|(\mathfrak{S}^{n_{1}}\widehat{\mu})_{i}(X)-p_i| = 0.
$$  
Therefore 
 \begin{equation}\label{e.quasequase}
 \lim _{n\to \infty}\int g \, d (\mathfrak{S}^{n}({\widehat{\mu}}))_{j}
=\int_{[j]} (g\circ\pi)\, d\,\mathbb{P}^{-}
\quad
\mbox{for every $g\in C^{0}(X)$}.
 \end{equation}
 To get equation \eqref{e.atra},
  write  $\widehat{f}=\langle f_{1},\dots,f_{k} \rangle$,
 apply  \eqref{e.quasequase}
 to the maps $f_{i}$,  and use \eqref{e.3101}
 to get
 $$
 \lim_{n\to \infty}\int \widehat f\, d \,\mathfrak{S}^{n}({\hat{\nu}})
 \underset{\tiny{\mbox{\eqref{e.3101}}}}{=}
\sum_{j=1}^{k}\lim_{n\to \infty}  \int f_{j}
\,d \, (\mathfrak{S}^{n}({\hat{\nu}}))_{j}
 \underset{\tiny{\mbox{\eqref{e.quasequase}}}}{=}
\sum_{j=1}^{k}\int_{[j]} (f_{j}\circ\pi)\,d\,\mathbb{P}^{-}.
$$
Now observing that $f_{j}\circ\pi(\xi)=\widehat f\circ \varpi (\xi)$ for every $\xi\in[j]$,
we conclude that 
$$
 \lim_{n\to \infty}\int \widehat f\, d \,\mathfrak{S}^{n}({\hat{\nu}})
 =\sum_{j=1}^{k}\int_{[j]} \widehat f\circ \varpi \, d\, \mathbb{P}^{-}=\int \widehat f\, 
 d\varpi_{*}\mathbb{P}^{-}.
 $$
 proving  \eqref{e.atra} and ending the proof of the theorem.
\end{proof}

In Proposition~\ref{p.p.recurrent}
we state a result that does not involve the mixing condition of  the probability $\mathbb{P}^-$.
For that we consider the subset  $\mathcal{M}_{\bar p}(\widehat{X})$ of
$\mathcal{M}_{1}(\widehat{X})$ defined by
$$
\mathcal{M}_{\bar p}(\widehat{X})\eqdef \{\widehat{\mu}=
(\mu_{1},\dots,\mu_{k})\colon \mu_{i}(X)=p_{i}\,\,\mbox{for every} \,\, i\},
$$
where $\bar p=(p_1,\dots, p_k)$ is the stationary vector of the irreducible transition matrix $P$ associated to
$\mathbb{P}^+$.
The set
$\mathcal{M}_{\bar p}(\widehat{X})$
 is invariant by $\mathfrak{S}_{\mathbb{P}^+}$
and
contains all stationary measures
of $\mathrm{IFS}(T_1,\dots,T_k; \mathbb{P}^+)$.
For the first assertion observe that
  given any $\widehat{\mu}\in  \mathcal{M}_{\bar p} (\widehat X)$
 by  the definition of  $\mathfrak{S}_{\mathbb{P}^{+}}$ we have
 $$
 (\mathfrak{S}_{\mathbb{P}^{+}}\widehat\mu)_j=\sum_{i=1}^{k}p_{ij}\, T_{j*}\mu_{i}
 \quad \mbox{for every $j$.}
 $$ 
 Thus 
 $$
 (\mathfrak{S}_{\mathbb{P}^{+}}\widehat\mu)_j(X)=\sum_{i=1}^{k}p_{ij}  \mu_{i}(T_{j}^{-1}(X))=\sum_{i=1}^{k}p_{ij}  \mu_{i}(X)=\sum_{i=1}^{k}p_{i}p_{ij} =p_{j}
 $$
 and hence $\mathfrak{S}_{\mathbb{P}^{+}} (\widehat{\mu})
 \in  \mathcal{M}_{\bar p} (\widehat X)$.

For the second assertion note that a measure  $\widehat{\mu}=(\mu_{1},\dots,\mu_{k})\in  \mathcal{M}_{1}(\widehat{X})$
is stationary if and only if
$$
\mu_{j}=\sum_{i=1}^{k}p_{ij}\, T_{j*}\mu_{i} \quad \mbox{for every} \,\, j.
$$
If $\widehat{\mu}=
(\mu_1,\dots,\mu_k)$  is stationary  then
$(\mu_{1}(X),\dots,\mu_{k}(X))$ is the stationary probability vector 
for the transition matrix $P$ of $\mathbb{P}^+$.  
Thus $\mu_{i}(X)=p_{i}$ for every $i$.
 
A corollary of Lemma~\ref{l.primitiveee} is the following proposition.  
 
\begin{prop}\label{p.p.recurrent}
Consider a recurrent  $\mathrm{IFS}(T_{1},\dots T_{k};\mathbb{P}^{+})$ defined on a compact 
metric space $X$ such that $\mathbb{P}^{-}(S_{\mathrm{t}})=1$. Then 
$$
\lim_{n\to \infty}\mathfrak{S}^{n}(\widehat{\nu})=\varpi_{*}\mathbb{P}^{-}
\quad
\mbox{for every
$\widehat{\nu}\in \mathcal{M}_{\bar p}(\widehat{X})$.}
$$
In particular, $\varpi_{*}\mathbb{P}^{-}$ is the unique stationary measure of
$\mathfrak{S}$.
\end{prop}

\begin{proof}
Consider $\widehat{\mu}=(\mu_1,\dots, \mu_k)\in \mathcal{M}_{\bar p}(\widehat{X})$ and note that 
 $\mu_i(X)=p_i$.
  Lemma
 \ref{l.primitiveee} implies that for every continuous function $g$ it holds 
 \begin{equation} \label{e.quasenoend}
 \lim_{n\to \infty} \int g \, d (\mathfrak{S}^{n}({\widehat{\mu}}))_{j}
=\int_{[j]} (g\circ\pi)\, d\,\mathbb{P}^{-}.
\end{equation}

Consider a continuous map
$\widehat{f}=\langle f_{1},\dots,f_{k} \rangle$.
 We apply \eqref{e.quasenoend} to the maps $f_{i}$ and use equation \eqref{e.3101}
 to get
 $$
 \lim_{n\to \infty}\int \widehat f\, d \,\mathfrak{S}^{n}({\hat{\nu}})
 \underset{\tiny{\mbox{\eqref{e.3101}}}}{=}
\sum_{j=1}^{k}\lim_{n\to \infty}  \int f_{j}
\,d \, (\mathfrak{S}^{n}({\hat{\nu}}))_{j}
 \underset{\tiny{\mbox{\eqref{e.quasenoend}}}}{=}
\sum_{j=1}^{k}\int_{[j]} (f_{j}\circ\pi)\,d\,\mathbb{P}^{-}.
$$
Observing that $f_{j}\circ\pi(\xi)=\widehat f\circ \varpi (\xi)$ for every $\xi\in[j]$,
we conclude that 
$$
 \lim_{n\to \infty}\int \widehat f\, d \,\mathfrak{S}^{n}({\hat{\nu}})
 =\sum_{j=1}^{k}\int_{[j]} \widehat f\circ \varpi \, d\, \mathbb{P}^{-}=\int \widehat f\, 
 d\varpi_{*}\mathbb{P}^{-},
 $$
 proving the proposition.
\end{proof}

\section{Examples}
\label{s.examples}


\begin{example}[A non-regular IFS with $S_{\mathrm{t}}\neq \emptyset$
and $\# (A_{\mathrm{t}})\ge 2$]
\label{ex.nonregular}
{\emph{
Consider an IFS defined on $[0,1]$ consisting
of two injective continuous maps $T_{1}$ and $T_{2}$ as in Figure \ref{f.nonregular}.
\begin{itemize}
\item
The map $T_{1}$ has  exactly
two fixed points $0,1$, where  $0$ is a repeller 
and $1$ is an attractor.
\item
The map $T_2$ 
has (exactly three) fixed points $p_{1}<p_{2}<p_{3}$,
where $p_1$ and $p_3$ are attractors and $p_2$ is a  repeller,
$T_{2}([0,1])=[\alpha,\beta]\subset (0,1)$, and
$T_1(p_1)< \beta$.
\end{itemize}
}}
 
{\emph{
Obviously, $\mathrm{IFS}(T_1)$
and $\mathrm{IFS}(T_2)$ 
are not asymptotically stable. 
To see that 
$\mathrm{IFS}(T_{1},T_{2})$ is not asymptotically stable 
just note that $[0,1]$ and $[p_{1},1]$ are fixed points of the Barnsley-Hutchinson operator. For the last assertion we use that $T_1(p_1)< \beta$. This implies that 
$\mathrm{IFS}(T_1,T_2)$ is non-regular.}}

\emph{Finally, to see that $S_{\mathrm{t}}\neq \emptyset$ note that  since 
$1$ is an attracting fixed point of $T_1$ and $T_{2}([0,1])\subset(0,1)$  we have that $T_{1}^{n}\circ T_{2}([0,1])\cap T_{2}([0,1])=\emptyset$
 for every $n$ sufficiently large. 
Now  Theorem \ref{t.separableee} implies that $S_{\mathrm{t}}\neq \emptyset$.
To see that  $\# (A_{\mathrm{t}})\ge 2$ just note that given any $x\in  A_{\mathrm{t}}$
then $T_i(x) \in  A_{\mathrm{t}}$ and that $T_1(x)\ne T_2(x)$}.
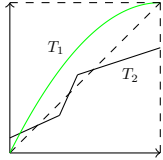
\begin{figure}[h!]
\centering
\begin{tikzpicture}[xscale=2,yscale=2]
\draw[->] (0,0)-- (1,0);
\draw[->] (0,0)--(0,1);
\draw[-]  (0,0.1)--(0.33,0.25);
\draw[-]    (0.33,0.25)--(0.45,0.52);
\draw[-] (0.45,0.52)--(1,0.7);
\draw[dashed] (0,0)--(1,1);
  \draw[green,smooth,samples=100,domain=0.0:1] plot(\x,{-\x*\x+2*\x});
 \node[scale=0.6,left] at (0.4,0.7) {$T_{1}$};
  \node[scale=0.6,below] at (0.8,0.6) {$T_{2}$};
\draw[dashed] (0,1)--(1,1);
\draw[dashed] (1,0)--(1,1);
\end{tikzpicture}
\caption{A non-regular IFS with a weakly hyperbolic sequence}
\label{f.nonregular}
\end{figure}
\end{example}

\begin{example}[{$A_{\mathrm{t}}\subsetneq \overline{A_{\mathrm{t}}}=[0,1]$}]
{\emph{In this example we consider 
 the underlying IFS of the porcupine-like
horseshoes in \cite{DiGe}. We translate the construction in \cite[page 12]{DiGeRams} to
our context.}}

{\emph{
Consider an injective $\mathrm{IFS}(T_{1},T_{2})$ defined on $[0,1]$ such that 
$T_{1}(x)=\lambda \, (1-x)$, $\lambda\in (0,1)$, and $T_{2}$ is a continuous function with exactly two fixed points,
the repelling fixed point $0$
 and the attracting 
fixed point $1$, see Figure \ref{Porcupine}.
 We assume that $T_{2}$ is a uniform contraction 
 on $[T_{2}^{-1}(\lambda),1]$. Then 
 $\overline{A_{\mathrm{t}}}=[0,1]$ and
$1\notin A_{\mathrm{t}}$.}}

{\emph{
  To prove the first assertion note that $\lambda \in \overline{A_{\mathrm{t}}}$. 
  For that
 take an open neighbourhood $V\subset (0,1)$ of $\lambda$. Note that $T_{1}^{-1}(V)$ is a neighbourhood 
 of $0$. 
 Consider the fixed point
  $p=\frac{\lambda}{1+\lambda}\in (0,1)$ 
 of $T_{1}$ and  note that $p\in A_{\mathrm{t}}$. 
 Since $T_{2}^{n}(p)\to 1$ as $n\to \infty$ and $T_{1}(1)=0$, there is $\ell$ such that 
 $T_{1}\circ T_{2}^{\ell}(p)\in T_{1}^{-1}(V)$. Hence $T_{1}^{2}\circ T_{2}^{\ell}(p)\in V$.
 By the invariance of $A_{\mathrm{t}}$ we have that $A_{\mathrm{t}}\cap V\neq \emptyset$.
 Since this holds for every neighbourhood  $V$ of $\lambda$ we get $\lambda \in \overline{A_{\mathrm{t}}}$.
 }}

{\emph{
 We now  prove that $A_{\mathrm{t}}$ is dense in $[0,1]$. Take any open interval $J\subset (0,1)$. We need to  see that 
 $J\cap A_{\mathrm{t}}\neq \emptyset$. If $\lambda\in J$ we are done. Otherwise $\lambda\notin J$
 and
 either $J\subset (\lambda,1]=I_{2}$ or $J\subset [0,\lambda)=I_{1}$. 
We now construct a finite sequence $\xi_{0}\dots \xi_{m}$ such 
that 
$$
\lambda \in T_{\xi_{m}}^{-1}\circ\dots\circ T_{\xi_{0}}^{-1}(J).
$$
For that let $\xi_{0}=i$ if $J\subset I_{i}$ and define recursively 
$\xi_{\ell+1}=i$ if $T_{\xi_{\ell}}^{-1}\circ\dots\circ T_{\xi_{0}}^{-1}(J)\subset I_{i}$.
Note that if 
$T_{\xi_{\ell}}^{-1}\circ\dots\circ T_{\xi_{0}}^{-1}(J)\cap I_{i} \ne\emptyset$ 
and $T_{\xi_{\ell}}^{-1}\circ\dots\circ T_{\xi_{0}}^{-1}(J)\cap I_{i} \nsubseteq I_i$ 
 some  $i=1,2$, we are done.
Since $T_{2}^{-1}$ is a uniform expansion on $(\lambda,1]$ and $T_{1}^{-1}$ is a uniform 
expansion on $[0,\lambda]$ the recursion stops after a finitely many steps:
there is $m$ such that $\lambda\in T_{\xi_{m}}^{-1}\circ\dots\circ T_{\xi_{0}}^{-1}(J)$.
Since $\lambda \in \overline{A_{t}}\cap T_{\xi_{m}}^{-1}\circ\dots\circ T_{\xi_{0}}^{-1}(J)$,
the invariance of $A_{\mathrm{t}}$ implies that $J\cap A_{\mathrm{t}}\neq \emptyset$.}}

\emph{The fact that $1\notin A_{\mathrm{t}}$ follows observing that 
$\bar 2 \not\in S_{\mathrm{t}}$ and that
every finite 
sequence $\xi_{0}\dots \xi_{n}$ such that $\xi_{i}=1$ for some $i$ 
satisfies $1\notin T_{\xi_{0}}\circ\dots\circ T_{\xi_{n}}([0,1])$}.

\begin{figure}[h!]
\centering
\begin{tikzpicture}[xscale=2,yscale=2]
\draw[->] (0,0)-- (1,0);
\draw[->] (0,0)--(0,1);
\draw[red,-] (0,0.9)--(1,0);
\draw[dashed] (0,0)--(1,1);
  \draw[green,smooth,samples=100,domain=0.0:1] plot(\x,{-\x*\x+2*\x});
 \node[scale=0.6,left] at (0.6,0.9) {$T_{2}$};
  \node[scale=0.6,below] at (0.8,0.4) {$T_{1}$};
\draw[dashed] (0,1)--(1,1);
\draw[dashed] (1,0)--(1,1);
\end{tikzpicture}
\caption{The underlying IFS of a porcupine-like horseshoe}
\label{Porcupine}
\end{figure}
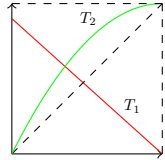
\end{example}

\begin{example}[A non-weakly hyperbolic IFS in
{$[0,1]$} with {$A_{\mathrm{t}}=[0,1]$}]
{\emph{We consider 
 the underlying IFS of the bony attractors in \cite{Ku}.}}

{\emph{Consider the $\mathrm{IFS}(T_{1},T_{2})$ defined on $[0,1]$ as follows,
 $T_{1}$ is the piecewise-linear map
with ``vertices'' $(0, 0)$, $(0.6, 0.2)$, and $(1, 0.8)$
and $T_{2}$ is the piecewise-linear map with ``vertices''
$(0, 0.15)$, $(0.4, 0.8)$, and $(1, 1)$, see Figure~\ref{bony}.
We claim that the
$\mathrm{IFS}(T_{1},T_{2})$ is not weakly hyperbolic and
$A_{\mathrm{t}}=[0,1]$.
}}

\begin{figure}[h!]
\centering
\begin{tikzpicture}[xscale=2,yscale=2]
 \node[scale=0.6,left] at (0.3,0.75) {$T_{2}$};
  \node[scale=0.6,below] at (0.7,0.6) {$T_{1}$};
\draw[->] (0,0)-- (1,0);
\draw[->] (0,0)--(0,1);
\draw[red,-]  (0,0)--(0.6,0.2);
\draw[red,-]  (0.6,0.2)--(1,0.75);
\draw[blue,-]  (0,0.15)--(0.4,0.8);
\draw[blue,-]  (0.4,0.8)--(1,1);
\draw[dashed] (0,1)--(1,1);
\draw[dashed] (0,0)--(1,1);
\draw[dashed] (1,0)--(1,1);
\end{tikzpicture}
\caption{The underlying IFS of a bony attractor}
\label{bony}
\end{figure}
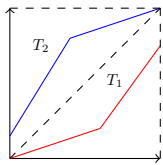

\emph{
To prove the first assertion note that $T_{1}\circ T_{2}$ has a repelling fixed point,
see \cite{Ku}.
Therefore the periodic sequence $\overline{12}$ does not belong to $S_{\mathrm{t}}$, 
hence the IFS is not weakly hyperbolic.}

\emph{
To see the second assertion, note that the compositions $T_{1}^3$, $T_{1}^{2}\circ T_{2}$, $T_{2}^{2}\circ T_{1}$ and $T_{2}^{5}$ are uniform contractions
and that the union of their images is $[0,1]$, see \cite{Ku}. In other words, 
the $\mathrm{IFS}(T_{1}^3, T_{1}^{2}\circ T_{2}, T_{2}^{2}\circ T_{1},T_{2}^{5})$ is hyperbolic 
and $[0,1]$ is the unique fixed point of its Barnsley-Hutchinson operator.
Consider the finite set of words 
$$
W=\{111,112,221, 22222\}
$$ 
and 
 let $E_{W}$ be the subset of $\Sigma_{k}^{+}$ consisting 
 of sequences $\xi$ that are a concatenation 
 of words of $W$\footnote{There is an increasing sequence $(i_{\ell})_{\ell\in \mathbb{N}}$
 with  $\xi_{0}=0$ such that $\xi_{i_{\ell}}\dots\xi_{i_{\ell+1}-1}\in W$ for every $\ell\in \mathbb{N}$.}. 
 Let $S_{\mathrm{t}}$ be the set of weakly hyperbolic sequences
 corresponding to the $\mathrm{IFS}(T_{1},T_{2})$ and
 $\pi$ the associated coding map. By construction we have that $E_{W}\subset S_{t}$ and 
$\pi(E_{W})=[0,1]$. Since $A_{t}=\pi(S_{\mathrm{t}})$ we have that 
$A_{\mathrm{t}}=[0,1]$}.

\end{example}

%
%
%
%

%
%
%
%
%
%
%

%

%


\end{document}